\documentclass[10pt,oneside,british,english]{amsart}
\usepackage[T1]{fontenc}
\usepackage[latin9]{inputenc}
\usepackage[a4paper]{geometry}
\usepackage{amsthm}
\usepackage{amstext}
\usepackage{amssymb}

\makeatletter
\numberwithin{equation}{section}
\numberwithin{figure}{section}
\theoremstyle{plain}
\newtheorem{thm}{\protect\theoremname}[section]
  \theoremstyle{plain}
  \newtheorem{prop}[thm]{\protect\propositionname}
  \theoremstyle{plain}
  \newtheorem{lem}[thm]{\protect\lemmaname}
  \theoremstyle{remark}
  \newtheorem{rem}[thm]{\protect\remarkname}
  \theoremstyle{plain}
  \newtheorem{cor}[thm]{\protect\corollaryname}
  \theoremstyle{plain}
  \newtheorem{assumption}[thm]{\protect\assumptionname}
  \theoremstyle{remark}
  \newtheorem*{acknowledgement*}{\protect\acknowledgementname}

\newcommand{\rxi}{\xi^{*}}
\newcommand{\rom}{\Omega^{*}}
\newcommand{\ru}{u^{*}}
\newcommand{\rp}{p^{*}}
\newcommand{\rU}{U^{*}}

\newcommand{\lxi}{\hat{\xi}}
\newcommand{\lom}{\hat{\Omega}}
\newcommand{\lu}{\hat{u}}
\newcommand{\lp}{\hat{p}}
\newcommand{\lU}{\hat{U}}

\newcommand{\fxi}{\xi^R}
\newcommand{\fom}{\Omega^R}
\newcommand{\fu}{u^R}

\newcommand{\fU}{U^R}

\newcommand{\Ve}{\Vert}

\makeatother

\usepackage{babel}
  \addto\captionsbritish{\renewcommand{\acknowledgementname}{Acknowledgement}}
  \addto\captionsbritish{\renewcommand{\assumptionname}{Assumption}}
  \addto\captionsbritish{\renewcommand{\corollaryname}{Corollary}}
  \addto\captionsbritish{\renewcommand{\lemmaname}{Lemma}}
  \addto\captionsbritish{\renewcommand{\propositionname}{Proposition}}
  \addto\captionsbritish{\renewcommand{\remarkname}{Remark}}
  \addto\captionsbritish{\renewcommand{\theoremname}{Theorem}}
  \addto\captionsenglish{\renewcommand{\acknowledgementname}{Acknowledgement}}
  \addto\captionsenglish{\renewcommand{\assumptionname}{Assumption}}
  \addto\captionsenglish{\renewcommand{\corollaryname}{Corollary}}
  \addto\captionsenglish{\renewcommand{\lemmaname}{Lemma}}
  \addto\captionsenglish{\renewcommand{\propositionname}{Proposition}}
  \addto\captionsenglish{\renewcommand{\remarkname}{Remark}}
  \addto\captionsenglish{\renewcommand{\theoremname}{Theorem}}
  \providecommand{\acknowledgementname}{Acknowledgement}
  \providecommand{\assumptionname}{Assumption}
  \providecommand{\corollaryname}{Corollary}
  \providecommand{\lemmaname}{Lemma}
  \providecommand{\propositionname}{Proposition}
  \providecommand{\remarkname}{Remark}
\providecommand{\theoremname}{Theorem}

\begin{document}

\title{Asymptotic behaviour of a rigid body with
a cavity filled by a viscous liquid}


\thanks{The author was supported by ERC-2010-AdG no.267802 \textquotedblleft{}Analysis
of Multiscale Systems Driven by Functionals\textquotedblright{}.}


\author{Karoline Disser}
\begin{abstract}
We consider the system of equations modeling the free motion of a
rigid body with a cavity filled by a viscous (Navier-Stokes) liquid.
We give a rigorous proof of Zhukovskiy's Theorem \cite{Zhukovskiy1885},
which states that in the limit $t\to\infty$, the relative fluid velocity
tends to zero and the rigid velocity of the full structure tends to
a steady rotation around one of the principle axes of inertia. 

The existence of global weak solutions for this system was established
in \cite{SilvestreTakahashi2012}. In particular, we prove that every
weak solution of this type is subject to Zhukovskiy's Theorem. Independently
of the geometry and of parameters, this shows that the presence of
fluid prevents precession of the body in the limit. In general, we
cannot predict which axis will be attained, but we show stability
of the largest axis and provide criteria on the initial data which
are decisive in special cases.
\end{abstract}
\maketitle

\section{Introduction\label{sec:Introduction}}

\global\long\def\tom{\tilde{\Omega}}
\global\long\def\io{\bar{\Omega}}
\global\long\def\oO{\Omega_{\infty}}
We consider a system of equations describing the motion of a rigid
body with a cavity filled by a viscous liquid. Let $\mathcal{S}\subset\mathbb{R}^{3}$
be a bounded closed domain which consists of a rigid body part $\mathcal{B}$,
which is also closed, and an open connected cavity domain $\mathcal{F}$
which contains the fluid. In particular, $\mathcal{S}=\mathcal{B}\cup\mathcal{F}$
and there is no {}``leak'', i.e. $\overline{\mathcal{F}}\cap\partial\mathcal{S}=\emptyset$.
We assume that the boundary $\Gamma=\mathcal{B}\cap\overline{\mathcal{F}}$
of $\mathcal{F}$ is of class $C^{2,1}$. We impose no further restrictions
on the geometry of $\mathcal{S}$. 

Without loss of generality, we assume that the fluid has density $\rho_{F}=1$
but the body's density is given by $\rho_{B}(y)>0$, $y\in\mathcal{B}$.
With $\mathcal{S}$ we associate the inertia tensor $I$ given by
\[
a^{T}\mathrm{I}b=\int_{\mathcal{F}}((y-y_{c})\times a)\cdot((y-y_{c})\times b)\,\mathrm{d}y+\int_{\mathcal{B}}((y-y_{c})\times a)\cdot((y-y_{c})\times b)\rho_{B}(y)\,\mathrm{d}y,
\]
for all $a,b\in\mathbb{R}^{3}$, where $y_{c}$ denotes the center
of mass of $\mathcal{S}$. We provide more details on modeling in
Section \ref{sec:Notation-and-MR}. In the absence of external forces
or torques, the equations for the coupled motion of the fluid and
the rigid body are given by
\begin{equation}
\left\{ \begin{array}{rcll}
\bar{u}'+\Omega'\times y-\nu\Delta\bar{u}+\nabla\bar{p}+2\Omega\times\bar{u}+(\bar{u}\cdot\nabla)\bar{u} & = & 0, & \text{in }(0,\infty)\times\mathcal{F},\\
\mathrm{div}\,\bar{u} & = & 0, & \text{in }(0,\infty)\times\mathcal{F},\\
\bar{u} & = & 0, & \text{on }(0,\infty)\times\Gamma,\\
I\io'+\Omega\times I\io & = & 0, & t\in(0,\infty),
\end{array}\right.\label{eq:A-1-1-1}
\end{equation}
where $\nu$ is the viscosity, $\mbox{\ensuremath{\bar{u}}}$ is the
\emph{relative }velocity of the fluid, $\bar{p}$ a pressure potential
and the angular velocity $\Omega$ of $\mathcal{B}$ and $\io$ are
related via 
\begin{equation}
\Omega=\io-I^{-1}\int_{\mathcal{F}}y\times\bar{u}(y)\,\mathrm{d}y.\label{eq:OmBarOm}
\end{equation}
In this frame of reference, the fluid is driven by $\Omega'\times y$
and the Coriolis term $2\Omega\times\bar{u}$ and the rigid body dynamics
are given by Euler's equations with a {}``fluid contribution''. 

In order to state our main result, we refer to Assumption \ref{Assumption on WS}
below, which asks that a weak solution for problem \eqref{eq:A-1-1-1}
satisfies conservation of momenta, the strong energy inequality and
weak-strong uniqueness. 
\begin{thm}
\label{thm:MR}Let $(\bar{u},\Omega)$ be a weak solution for \eqref{eq:A-1-1-1}
on $(0,\infty)\times\mathcal{S}$, satisfying Assumption \ref{Assumption on WS}.
Then $\Vert\bar{u}(t)\Vert_{H^{1}(\mathcal{F})}\to0$ and $\Omega(t)\to\oO$
as $t\to\infty$, where $\oO\in\mathbb{R}^{3}$ is a constant eigenvector
of $I$. 
\end{thm}
The claim of this result goes back to Zhukovskiy \cite{Zhukovskiy1885},
and we recall his argument from \cite[Chapter 2.2]{MoiseyevRumyantsev1968}:
Since the relative fluid motion dissipates kinetic energy, it must
come to rest as $t\to\infty$. If we plug $\bar{u}=0$ into \eqref{eq:A-1-1-1},
then in the first line, only 
\[
\Omega'\times y=-\nabla\bar{p}
\]
remains, where $\Omega'\times y$ has a potential only if $\Omega'=0$.
With $\Omega=\io$ from \eqref{eq:OmBarOm}, line 4 then implies that
$\Omega$ is an eigenfunction of $I$.

This argument shows immediately that the liquid part cannot pretend
to be rigid in general and that it excludes precession and the general
Poinsot solutions to Euler's equations 
\[
I\io'+\io\times I\io=0
\]
for the full structure. Note that in the presence of some dissipating
mechanism, this reasoning would also hold for an inviscid fluid (i.e.
the coupling of Euler's equation with the Euler equations). The aim
of this paper is to give a rigorous proof of Zhukovskiy's argument
for the viscous case.

There is a broad background on this problem in the engineering literature
and we point to the monographs \cite{MoiseyevRumyantsev1968} and
\cite{KopachevskyKrein2003} for many more references, in which the
main mathematical issue is the one of stability/instability of solutions
for special geometries. A natural application of the model is the
interpretation of the precession and nutation of planet earth, treated
for example by Poincaré \cite{Poincare1910}, and in \cite{StewartsonRoberts1963},
\cite{GreenspanHoward1963}.

In mathematical analysis, there is an extensive literature on the
\emph{complement problem} of the movement of a free rigid body immersed
in a fluid. We refer to \cite{Galdi02} for a survey of this topic,
and to \cite{GaldiSilvestre02}, \cite{Feireisl03}, \cite{GGH10a}
and \cite{Takahashi03} for additional existence and regularity theory.
In the absence of external forces, this system is dissipative and
both body and fluid must approach the rest state \cite{CumsilleTakahashi08},
but under the influence of external forces like gravity, many questions
regarding asymptotics are still open. 

For \eqref{eq:A-1-1-1}, global existence of weak solutions and local
existence of mild and strong solutions were proved in \cite{SilvestreTakahashi2012}.
We draw on their results and prove that these solutions satisfy the
assumptions of Theorem \ref{thm:MR}. A result similar to ours as
well as numerical studies of the problem were announced in \cite{GaldiMazzoneZunino2013}.

In \cite{LyashenkoFriedlander1998}, local existence of regular solutions
for the inviscid problem was proved. However, the main results of
\cite{LyashenkoFriedlander1998} and \cite{LyashenkoFriedlander1998B}
are explicit criteria for the derivation of non-linear instability
from linear instability in the inviscid limit. In particular, an instability
result for uniform rotation around one axis is proved also for small
viscosity. These results need a symmetry assumption for the structure
around the unstable axis and they are based on spectral linear stability
and instability for special geometries shown by Lyashenko \cite{Lyashenko1998},
\cite{Lyashenko1998B} and studied in \cite{KopachevskyKrein2003}.
Lyashenko and Friedlander's work directly addresses what is the main
feature of this system from a mathematical point of view: It is given
by the strong coupling of a non-dissipative and a dissipative part
with limited access to the actual (Navier-Stokes) dissipation. 

The general problem of specifying the limit angular velocity $\oO$
from initial data thus seems to be very difficult and may not be solvable
on the level of weak solutions. In Section \ref{sec:A-priori-characterizations},
we combine the conservation of total angular momentum and dissipation
of kinetic energy for this system in a very simple argument to define
an open subset of initial data which will always approach the largest
axis, proving stability in this sense. Making the set larger, we can
still show that the smallest axis will not be attained from any of
the initial data it contains. However, these estimates are crude,
not depending on viscosity or the actual dissipation of energy and
not sufficient for showing instability, e.g. of the {}``middle''
axis, which might be expected from classical rigid body dynamics \cite[Thm. 15.3.1]{MarsdenRatiu1999Book}. 

The outline of the proof and the organization of the paper is as follows.
Sections \ref{sec:Notation-and-MR} to \ref{sec:Global-in-time-existence}
mostly recount known results which are needed later on. In Section
\ref{sec:Notation-and-MR}, we introduce the model, fix some notation
and recall the change of coordinates to a Lagrangian formulation (with
respect to the rigid body). In Section \ref{sec:Local-in-time-existence},
we prove existence and continuous dependence on the data for local-in-time
strong solutions. We recall the weak formulation and existence proof
for global solutions given in \cite{SilvestreTakahashi2012} in Section
\ref{sec:Global-in-time-existence}.

In Section \ref{sec:Conservation-of-Momenta}, we show that every
weak solution given in \cite{SilvestreTakahashi2012} satisfies conservation
of the total linear momentum and of the total angular momentum. Even
though this shows that the kinetic energy 
\[
E(t)=\Vert\bar{u}(t,\cdot)+\Omega\times\cdot\Vert_{L^{2}(\mathcal{S})}^{2}
\]
does not decay to zero in general in this system, we still want to
show $\bar{u}(t)\to0$. Since there is no stability, the usual uniform
estimates (in the initial data) for the Navier-Stokes problem do not
apply and we have to work {}``trajectory-wise''. We provide a preparatory
higher-order a priori estimate on $\bar{u}$ in Section \ref{sec:Properties-of-SDS}.
In Section \ref{sec:Properties-of-WS}, we prove that global weak
solutions constructed in \cite{SilvestreTakahashi2012} satisfy the
strong\emph{ }energy inequality and weak-strong uniqueness. In Section
\ref{sec:Strong-solutions-for-large-time}, these results are combined
in order to prove that every weak solution becomes strong eventually
and that the relative fluid velocity then goes to rest. 

Section \ref{sec:Application-of-LaSalle's} concerns the second part
of Zhukovskiy's argument and the asymptotics for $\Omega(t)$. We
show that the kinetic energy is a strict Lyapunov functional on regular
(large-time) trajectories and characterize the equilibrium set. We
apply a version of LaSalle's invariance principle in order to prove
Theorem \ref{thm:MR}. 

Finally, in Section \ref{sec:A-priori-characterizations}, we derive
simple criteria which characterize the limit axis in special cases.

\section{Model and Notation\label{sec:Notation-and-MR}}

In order to fix some notation and for technical reasons, we will first
derive the model in Eulerian coordinates $x\in\mathbb{R}^{3}$, differing
from \eqref{eq:A-1-1-1}. This implies that the positions $\mathcal{B}(t)$,
$\mathcal{F}(t)$ and $\mathcal{S}(t)$ depend on time, with $\mathcal{S}(0)=\mathcal{S}$.
The body's mass is given by $\mathrm{m}_{B}:=\int_{\mathcal{B}}\rho_{B}(x)\,\mathrm{d}x$
and its inertia tensor $J_{B}(t)$ is given by 
\[
aJ_{B}(t)b=\int_{\mathcal{B}(t)}a\times(x-x_{B}(t))\cdot b\times(x-x_{B}(t))\rho_{B(t)}(x)\,\mathrm{d}x\quad\text{for all }a,b\in\mathbb{R}^{3}.
\]
The fluid motion is governed by the Navier-Stokes equations, driven
by an initial velocity and a no-slip boundary condition at $\Gamma(t)$,
where fluid and rigid body velocity must thus coincide. The rigid
body's center of mass 
\[
x_{B}(t)=\frac{1}{\mathrm{m}_{B}}\int_{\mathcal{B}(t)}x\rho_{B(t)}(x)\,\mathrm{d}x
\]
has a translational velocity $\eta$ and the body rotates with an
angular velocity $\omega$ with respect to $x_{B}$. It is driven
by its initial velocity $\eta_{0}+\omega_{0}\times(x-x_{B}(0))$ and
by the force $\int_{\Gamma(t)}\mathbf{T}(v,q)n(t)\, d\sigma$ and
the torque $\int_{\Gamma(t)}(x-x_{c}(t))\times\mathbf{T}(v,q)n(t)\, d\sigma$
exerted by the fluid velocity $v$ and pressure $q$. Here, 
\[
\mathbf{T}(v,q)=2\nu D(v)-q\mathrm{Id}
\]
is the Newtonian fluid stress tensor given by a constant viscosity
$\nu>0$ and the symmetric part of the gradient $D(v)=\frac{1}{2}[(\nabla v)+(\nabla v)^{T}]$,
where $n(t,x)$ denotes the outer normal of $\mathcal{B}(t)$ at $x\in\Gamma(t)$.
The system may additionally be subject to external forces and torques
$l_{0}$, $l_{1}$ and $l_{2}$ and in full, the equations read 

\begin{equation}
\left\{ \begin{array}{rcll}
v'-\mathrm{div}\,\mathbf{T}(v,q)+(v\cdot\nabla)v & = & l_{0}, & \text{in }\mathcal{Q}_{\mathcal{F}},\\
\mathrm{div}\, v & = & 0, & \text{in }\mathcal{Q}_{\mathcal{F}},\\
v(t,x)-\omega(t)\times(x-x_{B}(t))-\eta(t) & = & 0, & \text{on }\mathcal{Q}_{\Gamma},\\
v(0) & = & u_{0} & \text{on }\mathcal{F},\\
\mathrm{m}_{B}\eta'+\int_{\Gamma(t)}\mathbf{T}(v,q)n(t)\, d\sigma & = & l_{1}, & t>0,\\
(J_{B}\omega)'+\int_{\Gamma(t)}(x-x_{c}(t))\times\mathbf{T}(v,q)n(t)\, d\sigma & = & l_{2}, & t>0,\\
\eta(0)=\xi_{0} & \text{and} & \omega(0)=\Omega_{0},
\end{array}\right.\label{eq:Problem}
\end{equation}
where $\mathcal{Q}_{\mathcal{F}}:=\{(t,x)\in(0,\infty)\times\mathbb{R}^{3}:x\in\mathcal{F}(t)\}$
and $\mathcal{Q}_{\Gamma}$ is defined accordingly. In order to replace
the non-cylindrical domain $\mathcal{Q}_{\mathcal{F}}$ with a cylindrical
one, we change coordinates to a Lagrangian formulation with respect
to the rigid body $\mathcal{B}$. In particular, $x_{B}$ and $J_{B}$
will become independent of time. Without loss of generality, we set
$x_{B}(0)=0$. This is a standard procedure for this type of problem,
but we quickly repeat the construction as there is the technical detail
of having to deal with three centers of mass, $x_{B},x_{F}$ and $x_{c}$
and three inertia tensors $I_{B},I_{F}$ and $I$ of body, fluid and
the full structure. In particular, note that new coordinates are chosen
with respect to $x_{B}$, but for $t\to\infty$, $x_{c}$ and $I$
are more relevant. 

Let $m(t)$ denote the skew-symmetric matrix satisfying $m(t)x=\omega(t)\times x$.
Note that in the following, we denote derivatives with respect to
time by $\omega',v',\dots$, regardless of whether they are full or
partial. We consider the differential equation
\[
\left\{ \begin{array}{rcll}
X'(t,y) & = & m(t)(X(t,y)-x_{B}(t))+\eta(t), & (t,y)\in(0,T)\times\mathbb{R}^{3},\\
X(0,y) & = & y, & y\in\mathbb{R}^{3}.
\end{array}\right.
\]
As $\mathrm{div}\,(m(t)X(t,y))=0$, its solution is of the form $X(t,y)=Q(t)y+x_{B}(t),$
with some matrix $Q(t)\in\mathrm{SO}(3)$ for every $t\in(0,T)$.
In particular, $Q\in H^{2}(0,T;\mathbb{R}^{3\times3})$, if $\eta,\omega\in H^{1}(0,T)$,
justifying this change of coordinates a posteriori for strong solutions.
The corresponding inverse $Y(t)$ of $X(t)$ is given by 
\[
Y(t,x)=Q^{T}(t)(x-x_{B}(t)).
\]
For $(t,y)\in[0,T)\times\mathbb{R}^{3}$, we thus define
\begin{eqnarray}
u(t,y) & := & Q^{T}v(t,X(t,y)),\nonumber \\
p(t,y) & := & q(t,X(t,y)),\nonumber \\
\Omega(t) & := & Q^{T}(t)\omega(t),\label{eq:DefDerTrafo}\\
\xi(t) & := & Q^{T}(t)\eta(t),\nonumber \\
f_{i}(t,y) & := & Q^{T}l_{i}(t,X(t,y)).\nonumber 
\end{eqnarray}
It follows that the transformed inertia tensor $I_{B}=Q^{T}(t)J(t)Q(t)$
for the rigid body part $\mathcal{B}$ no longer depends on time and
that for all $a,b\in\mathbb{R}^{3},$ 
\begin{equation}
aI_{B}b=\int_{\mathcal{B}}(a\times y)\cdot(b\times y)\,\rho_{B}(y)\mathrm{d}y.\label{eq:inertiaTensorTransformiert}
\end{equation}
The definition 
\[
T(u,p):=2\nu D_{y}(u)-\nabla_{y}p,
\]
where $D_{y},\nabla_{y}$ here explicitly indicate differentiation
with respect to the new coordinates $y$, implies that 
\[
\int_{\Gamma(t)}\mathbf{T}(v,q)n(t)\,\mathrm{d}\sigma=Q\int_{\Gamma}T(u,p)N\,\mathrm{d\sigma}
\]
and 
\[
\int_{\Gamma(t)}(x-x_{B}(t))\times\mathbf{T}(v,q)n(t)\,\mathrm{d}\sigma=Q\int_{\Gamma}y\times T(u,p)N\,\mathrm{d}\sigma.
\]
On the cylindrical domain $(0,T)\times\mathcal{F}$ with outer normal
vector $-N=-Q^{T}(t)n(t)$, we obtain the system of equations 
\begin{equation}
\left\{ \begin{array}{rcll}
u'-\mu\Delta u+\nabla p+\Omega\times u+((u-\xi-\Omega\times y)\cdot\nabla)u & = & f_{0}, & \text{in }(0,T)\times\mathcal{F},\\
\mathrm{div}\, u & = & 0, & \text{in }(0,T)\times\mathcal{F},\\
u(t,y)-\Omega(t)\times y-\xi(t) & = & 0, & \text{on }(0,T)\times\Gamma,\\
u(0) & = & u_{0} & \text{in }\mathcal{F},\\
\mathrm{m}_{B}\xi'+\mathrm{m}_{B}(\Omega\times\xi)+\int_{\Gamma}T(u,p)N\,\mathrm{d}\sigma & = & f_{1}, & t\in(0,T),\\
I_{B}\Omega'+\Omega\times(I_{B}\Omega)+\int_{\Gamma}y\times T(u,p)N\,\mathrm{d}\sigma & = & f_{2}, & t\in(0,T),\\
\xi(0)=\xi_{0} & \text{and} & \Omega(0)=\Omega_{0},
\end{array}\right.\label{eq:FFP transformiert}
\end{equation}
to be equivalent to \eqref{eq:Problem} with the unknowns $u,p$ the
new fluid velocity and pressure and $\xi,\Omega$ the rigid body's
translational and angular velocity. We denote the center of mass of
the fluid part $\mathcal{F}$ by 
\[
x_{F}(0)=y_{F}=\frac{1}{\mathrm{m}_{F}}\int_{\mathcal{F}}y\,\mathrm{d}y,
\]
and the center of mass of the full structure $\mathcal{S}$ by 
\begin{equation}
x_{c}(0)=y_{c}=\frac{1}{\mathrm{m}}\left[\int_{\mathcal{F}}y\,\mathrm{d}y+\int_{\mathcal{B}}y\rho_{B}(y)\,\mathrm{d}y\right]=\frac{m_{F}}{\mathrm{m}}y_{F},\label{eq:comO}
\end{equation}
where $m_{F}+m_{B}=\mathrm{m}$ is the total mass. We often use the
calculation rules
\begin{eqnarray}
a\times(b\times c) & = & b(a\cdot c)-c(a\cdot b),\label{eq:aTimesb}\\
(a\times b)\cdot(c\times d) & = & (a\cdot c)(b\cdot d)-(a\cdot d)(b\cdot c),\label{eq:aTimesbII}
\end{eqnarray}
for all $a,b,c\in\mathbb{R}^{3}$. To the full structure, we associate
the inertia tensor $I$ calculated with respect to the center of mass
$y_{c}$, 
\[
a^{T}\mathrm{I}b=\int_{\mathcal{F}}((y-y_{c})\times a)\cdot((y-y_{c})\times b)\,\mathrm{d}y+\int_{\mathcal{B}}((y-y_{c})\times a)\cdot((y-y_{c})\times b)\rho_{B}(y)\,\mathrm{d}y
\]
as in Section \ref{sec:Introduction}. Using \eqref{eq:inertiaTensorTransformiert},
\eqref{eq:comO} and \eqref{eq:aTimesbII}, it follows that 
\begin{equation}
\mathrm{I}b=(I_{B}+I_{F})b+\mathrm{m}y_{c}\times(y_{c}\times b),\label{eq:ItoIBIF}
\end{equation}
where $I_{F}$ is the inertia tensor of $\mathcal{F}$, calculated
with respect to the center of mass $0$ of the rigid part. 

In the following, to a triple $u,\xi,\Omega$ of solutions, we often
associate the function 
\begin{equation}
U(t,y):=\begin{cases}
\begin{array}{ll}
u(t,y), & y\in\mathcal{F},\\
\Omega(t)\times y+\xi(t), & y\in\mathcal{B},
\end{array}\end{cases}\label{eq:Identifications}
\end{equation}
and vice versa. To both, we associate the \emph{relative fluid velocity
\[
\bar{u}(t,y):=u(t,y)-\xi(t)-\Omega(t).
\]
}Note that it is shown below that if $U$ is a weak solution of \eqref{eq:FFP transformiert},
then 
\[
\bar{u}(t)\in H_{0}^{1}(\mathcal{F}):=\{u\in H^{1}(\mathcal{F});u|_{\Gamma}=0\}
\]
 for almost all $t$ and, in addition, $\bar{u}(t)\in L_{\sigma}^{2}(\mathcal{F})$,
where 
\[
L_{\sigma}^{2}(\mathcal{F}):=\{u\in L^{2}(\mathcal{F});\mathrm{div}\, u=0,\, u|_{\Gamma}\cdot n=0\,\text{in a weak sense}\}
\]
 is the usual space of solenoidal $L^{2}$-functions and $H^{1}(\mathcal{F})$
is the usual $L^{2}$-Sobolev space of order $1$. We define 
\[
H:=L_{\sigma}^{2}(\mathcal{F})\cap H_{0}^{1}(\mathcal{F})
\]
 and for $1\leq q\leq\infty$, $\Vert\cdot\Vert_{q}:=\Vert\cdot\Vert_{L^{q}(\mathcal{F})}$
denotes the $L^{q}$-norm on $\mathcal{F}$ . We often apply Poincaré's
inequality to $\bar{u}$ with constant $C_{p}$, $\Vert\bar{u}\Vert_{2}\leq C_{p}\Vert\nabla\bar{u}\Vert_{2}$.

\section{\label{sec:Local-in-time-existence}Local-in-time existence of strong
solutions}

For sufficiently regular solutions, it is required that the initial
data $u_{0},\xi_{0},\Omega_{0}$ satisfy the compatibility condition
\[
U_{0}\in\mathcal{W}:=\{(u_{0}.\xi_{0},\Omega_{0})\in H^{1}(\mathcal{F})\times\mathbb{R}^{6};\mathrm{div}\, u_{0}=0,\, u_{0}|_{\Gamma}(y)=\Omega_{0}\times y+\xi_{0}\},
\]
where we refer to \cite[Rem. 2.3c)]{GGH10a} for a discussion of this
constraint in this context. In particular, $\mathcal{W}$ is the time-trace
space for the strong solution and it follows that $\bar{u}_{0}\in H$. 
\begin{thm}
\label{thm:mainResultN-1}Let $U_{0}\simeq(u_{0},\xi_{0},\Omega_{0})\in\mathcal{W}$
and 
\[
F:=(f_{0},f_{1},f_{2})\in L^{2}(0,T_{0};L^{2}(\mathcal{F}))\times L^{2}(0,T_{0};\mathbb{R}^{6})=:\mathcal{V}^{T_{0}}
\]
 be given. Then there exists $0<T_{max}\leq T_{0}$ such that problem
\eqref{eq:Problem} admits a unique strong solution
\begin{eqnarray*}
u & \in & L^{2}(0,T;H^{2}(\mathcal{F}))\cap H^{1}(0,T;L^{2}(\mathcal{F}))=:X_{2,2}^{T}\\
\nabla p & \in & L^{2}(0,T;L^{2}(\mathcal{F})),\\
(\Omega,\xi) & \in & H^{1}(0,T;\mathbb{R}^{6}),
\end{eqnarray*}
for all $0<T<T_{max}$. Moreover, the solution in these spaces depends
continuously on the data $(U_{0},F)$ in $\mathcal{W}\times\mathcal{V}^{T_{max}}$. 
\end{thm}
We prove this result almost exactly as in the {}``complement case''
of a rigid body immersed in a viscous liquid (filling a bounded or
exterior domain). Note that for our exact situation, a proof was given
already in \cite{SilvestreTakahashi2012}, however, we need to recall
some arguments in order to justify continuous dependence on the data
and Corollary \ref{cor:Corollary} below. The proof here uses maximal
regularity-type estimates of the linearized problem in $L^{2}(\mathcal{F})\times\mathbb{R}^{3}\times\mathbb{R}^{3}$
and the contraction mapping principle. A suitable linearization of
\eqref{eq:FFP transformiert} is exactly the same as for the complement
problem, except that here, we do not need to include an additional
boundary condition at $\partial\mathcal{S}$, i.e. it is given by
\begin{equation}
\left\{ \begin{array}{rcll}
u_{t}-\Delta u+\nabla p & = & f_{0}, & \text{in }(0,T)\times\mathcal{F},\\
\mathrm{div}\, u & = & 0, & \text{in }(0,T)\times\mathcal{F},\\
\bar{u} & = & 0, & \text{on }(0,T)\times\Gamma,\\
u(0) & = & u_{0} & \text{in }\mathcal{F},\\
\mathrm{m}_{B}\xi'+\int_{\Gamma}T(u,p)N\,\mathrm{d}\sigma & = & f_{1}, & t\in(0,T),\\
I_{B}\Omega'+\int_{\Gamma}y\times T(u,p)N\,\mathrm{d}\sigma & = & f_{2}, & t\in(0,T),\\
\xi(0)=\xi_{0} & \text{and} & \Omega(0)=\Omega_{0}.
\end{array}\right.\label{eq:FullLinearProblem-1}
\end{equation}
Thus, we cite the following result from \cite[Thm. 4.1]{GGH10a}.
\begin{prop}
\label{thm:linearResult-1}For all $(U_{0},F)\in\mathcal{W}\times\mathcal{V}^{T_{0}}$,
there is a unique solution 
\begin{eqnarray*}
u & \in & X_{2,2}^{T_{0}}\\
\nabla p & \in & L^{2}(0,T_{0};L^{2}(\mathcal{F})),\\
(\xi,\Omega) & \in & H^{1}(0,T_{0};\mathbb{R}^{6})
\end{eqnarray*}
to \eqref{eq:FullLinearProblem-1}, which satisfies 
\[
\left\Vert u\right\Vert _{X_{2,2}^{T_{0}}}+\left\Vert \nabla p\right\Vert _{L^{2}(0,T_{0};L^{2}(\mathcal{F}))}+\left\Vert \xi\right\Vert _{H^{1}(0,T_{0})}+\left\Vert \Omega\right\Vert _{H^{1}(0,T_{0})}\leq C_{MR}(\left\Vert F\right\Vert _{\mathcal{V}^{T_{0}}}+\left\Vert U_{0}\right\Vert _{\mathcal{W}}),
\]
where the constant $C_{MR}$ depends on geometry and material parameters
and on $T_{0}$.
\end{prop}
In order to prove Theorem \ref{thm:mainResultN-1}, we rewrite \eqref{eq:FFP transformiert}
as a fixed point equation in $U$. We denote by $\ru,\rp,\rxi,\rom$
the unique solution of \eqref{eq:FullLinearProblem-1}. Let $\lu=u-\ru$,$\lp=p-\rp$,
$\lxi=\xi-\rxi$, $\lom=\Omega-\rom$ and $\hat{\bar{u}}=\bar{u}-\bar{u}^{*}$.
Then \eqref{eq:FFP transformiert} is equivalent to 
\begin{equation}
\left\{ \begin{array}{rcll}
\lu_{t}-\Delta\lu+\nabla\lp & = & \mathcal{R}_{0}(\lu,\lxi,\lom), & \text{in }(0,T)\times\mathcal{F},\\
\mathrm{div}\,\lu & = & 0, & \text{in }(0,T)\times\mathcal{F},\\
\hat{\bar{u}} & = & 0, & \text{on }(0,T)\times\Gamma,\\
\lu(0) & = & 0, & \text{in }\mathcal{F},\\
\mathrm{m}_{B}(\lxi)'+\int_{\Gamma}T(\lu,\lp)N\,\mathrm{d}\sigma & = & \mathcal{R}_{1}(\lxi,\lom), & \text{in }(0,T),\\
I_{B}(\lom)'+\int_{\Gamma}y\times T(\lu,\lp)N\,\mathrm{d}\sigma & = & \mathcal{R}_{2}(\lom), & \text{in }(0,T),\\
\lxi(0)=0 & \text{and} & \lom(0)=0,
\end{array}\right.\label{eq:FPAFormulierungN}
\end{equation}
where 
\begin{eqnarray*}
\mathcal{R}_{0}(\lu,\lxi,\lom) & = & (\lom+\rom)\times(\lu+\ru)-((\hat{\bar{u}}+\bar{u}^{*})\cdot\nabla)(\lu+\ru),\\
\mathcal{R}_{1}(\lxi,\lom) & = & \mathrm{m}(\lxi+\rxi)\times(\lom+\rom),\\
\mathcal{R}_{2}(\lom) & = & (\lom+\rom)\times I(\lom+\rom).
\end{eqnarray*}
Given a fixed $T_{0}$, we define
\[
C_{*}:=\left\Vert \rU\right\Vert _{\mathcal{X}^{T_{0}}}=C_{MR}(\left\Vert U_{0}\right\Vert _{\mathcal{W}}+\Vert f\Vert_{\mathcal{V}^{T_{0}}}).
\]
The solution should satisfy, for some $T\leq T_{0}$, 
\begin{eqnarray*}
\lu & \in & X_{2,2,0}^{T}:=\{\lu\in X_{2,2}^{T}:\lu|_{t=0}=0\},\\
\nabla\lp & \in & L^{2}(0,T;L^{2}(\mathcal{F}))\text{ and}\\
\lxi,\lom & \in & H_{0}^{1}(0,T):=\{\hat{r}\in H^{1}(0,T):\hat{r}|_{t=0}=0\},
\end{eqnarray*}
so we choose the ball $\mathcal{X}_{R}^{T}$ as 
\[
\mathcal{X}_{R}^{T}:=\{U\in X_{2,2,0}^{T}\times H_{0}^{1}(0,T;\mathbb{R}^{3})\times H_{0}^{1}(0,T;\mathbb{R}^{3}):\left\Vert U\right\Vert _{\mathcal{X}^{T}}\leq R\}.
\]
In \eqref{eq:sinnlos1} below, we see that $R:=C_{*}$ is optimal.
Let

\[
\phi_{R}^{T}:\fU\mapsto\left(\begin{array}{c}
\mathcal{R}_{0}(\fu,\fxi,\fom)\\
\mathcal{R}_{1}(\fxi,\fom)\\
\mathcal{R}_{2}(\fom)
\end{array}\right)\mapsto U
\]
be the function which maps $\fU\in\mathcal{X}_{R}^{T}$ to the solution
$U$ of the linear problem \eqref{eq:FullLinearProblem-1} with right
hand sides $\mathcal{R}_{0}(\fu,\fxi,\fom),\mathcal{R}_{1}(\fxi,\fom),\mathcal{R}_{2}(\fom)$
and initial value $U_{0}=0$. 

In the following, $C>0$ denotes a generic constant which may depend
on $T_{0}$, but can be chosen independently of $T,R$ for $0<T\leq T_{0}$.
Assume $\fU,\fU_{1},\fU_{2}\in\mathcal{X}_{R}^{T}$ and set $U=\fU+\rU$,
$U_{1}=\fU_{1}+\rU_{1}$ etc.

We use the estimates
\[
\left\Vert U\right\Vert _{\mathcal{X}^{T}}\leq R+C_{*}\;\text{and }\;\Vert\bar{u}\Vert_{X_{2,2}^{T}}\leq C\Vert U\Vert_{\mathcal{X}^{T}}
\]
and that $u\in X_{2,2}^{T}$ satisfies 
\begin{equation}
u\in BUC([0,T];H^{1}(\mathcal{F})),\label{eq:BUCeinbettung}
\end{equation}
which follows from \cite[III.4.10]{Amann95} and by the construction
of $u$ in \cite[Section 4]{GGH10a}. Here, $BUC([0,T)$ denotes the
space of bounded uniformly continuous functions on $[0,T]$ and we
note that the embedding constant does not depend on $T\leq T_{0}$
if it is restricted to the subspace $X_{2,2,0}^{T}$ of $X_{2,2}^{T}$.
It follows that
\begin{equation}
\Vert\nabla u\Vert_{2,\infty}+\Vert\nabla\bar{u}\Vert_{2,\infty}\leq C(R+C_{*}).\label{eq:RCstar}
\end{equation}
The function $\phi_{R}^{T}$ maps $\mathcal{X}_{R}^{T}$ into itself,
if it is shown that 
\begin{eqnarray*}
\Ve\phi_{R}^{T}(\fU)\Ve_{\mathcal{X}^{T}} & \leq & C_{MR}\bigl(\Ve\mathcal{R}_{0}(\fu,\fxi,\fom),\mathcal{R}_{1}(\fxi,\fom),\mathcal{R}_{2}(\fom)\Ve_{\mathcal{V}^{T}}\bigr)\leq R.
\end{eqnarray*}
The term $((\hat{\bar{u}}+\bar{u}^{*})\cdot\nabla)(\lu+\ru)$ in this
estimate is treated as usual for the Navier-Stokes problem, i.e. note
that $\bar{u}\in X_{2,2}^{T}$ if 
\[
U\simeq(u,\Omega,\xi)\in X_{2,2}^{T}\times H^{1}(0,T)\times H^{1}(0,T):=\mathcal{X}^{T}
\]
and that in addition, $\bar{u}=0$ on $\Gamma$. Given $u,\bar{u}\in X_{2,2}^{T}$
such that $\bar{u}\in H_{0}^{1}(\mathcal{F})$, we have by Hölder's
inequality, 
\[
\Vert(\bar{u}\cdot\nabla)u\Vert_{2,2}\leq\Vert\bar{u}\Vert_{2,\infty}\Vert\nabla u\Vert_{\infty,2}.
\]
Since $\mathcal{F}$ is bouded, we have $\Vert\bar{u}(t)\Vert_{\infty}\leq C(1+C_{p})\Vert\nabla\bar{u}\Vert_{3+\varepsilon}$
for every $\varepsilon>0$. The Riesz-Thorin Theorem yields $\Vert\nabla\bar{u}\Vert_{3+\varepsilon}\leq2\Vert\nabla\bar{u}\Vert_{2}^{\theta}\Vert\nabla\bar{u}\Vert_{6}^{(1-\theta)}$
where $\theta=\frac{1}{2}\frac{3-\varepsilon}{3+\varepsilon}$. In
conclusion,
\begin{eqnarray}
\Vert(\bar{u}\cdot\nabla)u\Vert_{2,2} & \leq & \Vert\bar{u}\Vert_{2,\infty}\Vert\nabla u\Vert_{\infty,2}\nonumber \\
 & \leq & C(\int_{0}^{T}\Vert\nabla\bar{u}(t)\Vert_{2}^{2\theta}\Vert\nabla\bar{u}(t)\Vert_{6}^{2(1-\theta)}\,\mathrm{d}t)^{1/2}\Vert\nabla u\Vert_{\infty,2}\nonumber \\
 & \leq & C\Vert\nabla u\Vert_{\infty,2}\Vert\nabla\bar{u}\Vert_{\infty,2}^{\theta}\Vert\nabla\bar{u}\Vert_{2(1-\theta),6}^{(1-\theta)}\nonumber \\
 & \leq & CT^{\theta/2}\Vert\nabla u\Vert_{\infty,2}\Vert\nabla\bar{u}\Vert_{\infty,2}^{\theta}\Vert D^{2}\bar{u}\Vert_{2,2}^{(1-\theta)}.\label{eq:VNablauEst}
\end{eqnarray}
In addition (not optimally), 
\[
\Vert\Omega\times u\Vert_{2,2}\leq\Vert\Omega\Vert_{2,\infty}\Vert u\Vert_{\infty,2}\leq CT^{1/2}(R+C_{*})^{2},
\]
so that by \eqref{eq:VNablauEst}, \eqref{eq:BUCeinbettung} and \eqref{eq:RCstar},
we obtain 
\begin{eqnarray*}
\Ve\mathcal{R}_{0}(\fu,\fxi,\fom)\Ve_{L^{2}(0,T;L^{2}(\mathcal{F}))} & \leq & C(T^{1/2}+T^{\theta/2})(R+C_{*})^{2},
\end{eqnarray*}
and (again not optimally),
\begin{eqnarray*}
\Ve\mathcal{R}_{1}(\fxi,\fom)\Ve_{L^{2}(0,T)}+\Ve\mathcal{R}_{2}(\fom)\Ve_{L^{2}(0,T)} & \leq & CT^{1/2}(R+C_{*})^{2}.
\end{eqnarray*}
Thus, $\phi_{R}^{T}$ maps into $\mathcal{X}_{R}^{T}$ if 
\begin{equation}
T^{1/2}+T^{\theta/2}\leq\frac{R}{CC_{MR}(R+C_{*})^{2}}.\label{eq:sinnlos1}
\end{equation}
A simple argument shows that $R=C_{*}$ maximizes $T$ with $T(C_{*})^{1/2}+T(C)_{*}^{\theta/2}=\frac{1}{4CC_{MR}C_{*}}$.
Via the estimate
\begin{eqnarray*}
\Ve\phi_{R}^{T}(\fU_{1})-\phi_{R}^{T}(\fU_{2})\Ve_{\mathcal{X}^{T}} & \leq & CC_{*}(T^{1/2}+T^{\theta/2})\Ve\fU_{1}-\fU_{2}\Ve_{\mathcal{X}^{T}},
\end{eqnarray*}
the map $\phi_{C_{*}}^{T}$ is contractive for small $T$. The contraction
mapping theorem yields a unique fixed point $\lU$ of $\phi_{C_{*}}^{T}$
which gives a corresponding pressure $\lp$ and a strong solution
$(\lu,\lp,\lxi,\lom)$ of problem \eqref{eq:FPAFormulierungN}. At
the same time, we can deduce continuous dependence of the solution
on the data in the following sense. Given $U_{0},V_{0}\in\mathcal{W}$
and $F,G\in\mathcal{V}^{T_{0}}$, there are solutions $U^{*},V^{*}\in\mathcal{X}^{T_{0}}$
of the linear problem \eqref{eq:FullLinearProblem-1} and solutions
$U\in\mathcal{X}^{T(C_{*}(U_{0},f))}$, $V\in\mathcal{X}^{T(C_{*}(V_{0},g))}$
of \eqref{eq:FFP transformiert}. Their difference can be estimated
by 
\begin{eqnarray*}
\Vert U-V\Vert_{\mathcal{X}^{T}} & \leq & C_{MR}\Vert\mathcal{R}(U+U^{*})-\mathcal{R}(V+V^{*})\Vert_{2,2\times L^{2}(0,T)\times L^{2}(0,T)}\\
 & \leq & CC_{MR}\max(C_{*})(T^{1/2}+T^{\theta/2})(\Vert U-V\Vert_{\mathcal{X}^{T}}+\Vert U^{*}-V^{*}\Vert_{\mathcal{X}^{T}}),
\end{eqnarray*}
where $T=\min(T(C_{*}))$. Since $CC_{MR}\max(C_{*})(T^{1/2}+T^{\theta/2})<1$,
and $U^{*}$ and $V^{*}$ depend linearly on the data, we obtain 
\[
\Vert U-V\Vert_{\mathcal{X}^{T}}\leq C(T,\max(C_{*}))\left(\Vert U_{0}-V_{0}\Vert_{\mathcal{W}}+\Vert F-G\Vert_{\mathcal{V}^{T_{0}}}\right).
\]
The maximal time $T_{max}$ of existence for these solutions is characterized
as follows. Either $T_{max}=\infty$ or one of the functions 
\begin{equation}
t\mapsto\Vert u(t)\Vert_{H^{1}},\quad t\mapsto\vert\Omega(t)\vert,\quad t\mapsto\vert\xi(t)\vert\label{eq:BlowUp}
\end{equation}
blows up as $t\to T_{max}$, because otherwise, the solution could
be extended. In \eqref{eq:trueBlowUp} below, we show that this condition
is equivalent to the blow-up of $t\mapsto\Vert\bar{u}(t)\Vert_{H}.$
Note that continuous dependence on the data extends to $T_{max}$.
This proves Theorem \ref{thm:mainResultN-1}.

\section{\label{sec:Global-in-time-existence}Global-in-time existence of
weak solutions}

From here, we assume that there are no external forces or torques
driving the system, i.e. $f_{0},f_{1},f_{2}=0$. We cite from \cite[Def. 5.5, Thm. 5.6]{SilvestreTakahashi2012}
the global existence of weak solutions for \eqref{eq:FFP transformiert}.
Let us first introduce some notation. We write 
\begin{eqnarray*}
\mathcal{L} & := & \{U\in L^{2}(\mathcal{S}):\mathrm{div}\, U=0\text{ in }\mathcal{S},D(U)=0\text{ in }\mathcal{B}\}\\
\mathcal{H} & := & \{U\in H^{1}(\mathcal{S}):\mathrm{div}\, U=0\text{ in }\mathcal{S},D(U)=0\text{ in }\mathcal{B}\}.
\end{eqnarray*}
By $\langle\cdot,\cdot\rangle_{\mathcal{H}}$, we denote the duality
product between $\mathcal{H}$ and $\mathcal{H}'$. For a detailed
discussion and characterization of these spaces, we refer to \cite[Sections 3,4]{SilvestreTakahashi2012}.
Here, we note that each $U\in\mathcal{L}$ can be characterized as
\begin{equation}
U|_{\mathcal{B}}(y)=\Omega_{U}\times y+\xi_{U}\label{eq:Ubreakup}
\end{equation}
 for some $\Omega_{U},\xi_{U}\in\mathbb{R}^{3}$, so that the identifications
in \eqref{eq:Identifications} apply. Moreover, we note that $L^{2}(\mathcal{S})$
is endowed with the measure $\rho\,\mathrm{d}y$, where 
\[
\rho(y)=\begin{cases}
\begin{array}{ll}
\rho_{B}(y), & y\in\mathcal{B},\\
1, & y\in\mathcal{F}.
\end{array}\end{cases}
\]
The symmetric continuous bilinear form $a:\mathcal{H}\times\mathcal{H}\to\mathbb{R}$
is given by
\[
a(U,W):=2\nu\int_{\mathcal{F}}D(U):D(W)\,\mathrm{d}y
\]
and the trilinear form $b:\mathcal{H}\times\mathcal{H}\times\mathcal{H}\to\mathbb{R}$
is given by 
\[
b(U,V,W):=m_{B}(\Omega_{V}\times\xi_{U})\cdot\xi_{W}+\Omega_{V}\times I_{B}\Omega_{U}\cdot\Omega_{W}+\int_{\mathcal{F}}(\bar{v}\cdot\nabla u)\cdot w\,\mathrm{d}y+\int_{\mathcal{F}}\Omega_{V}\times u\cdot w\,\mathrm{d}y.
\]

\begin{thm}
\label{thm:ExWeak}Let $U_{0}\in\mathcal{L}$. Then there exists a
$U\in L^{\infty}(0,\infty;\mathcal{L})\cap L_{loc}^{2}([0,\infty);\mathcal{H})\cap C_{w}([0,\infty);\mathcal{L})$
with $U'\in L^{4/3}([0,\infty);\mathcal{H}')$ such that for all $\phi\in\mathcal{H}$,
\begin{equation}
\langle U',\phi\rangle+a(U,\phi)+b(U,U,\phi)=0\quad\text{a.e. in }(0,\infty),\label{eq:WeakFormula}
\end{equation}
and $U(0)=U_{0}$ is attained in the weak sense. In particular, for
all $t>0$, $U$ satisfies the energy inequality
\begin{equation}
\Vert U(t)\Vert_{L^{2}(\mathcal{S})}^{2}+4\nu\int_{0}^{t}\Vert D(u)(s)\Vert_{2}^{2}\,\mathrm{d}s\leq\Vert U_{0}\Vert_{L^{2}(\mathcal{S})}^{2}.\label{eq:EI}
\end{equation}

\end{thm}
Note that by a direct calculation and the identifications in \eqref{eq:Identifications},
every strong solution $u,\xi,\Omega$ given by Theorem \ref{thm:mainResultN-1}
provides a weak solution $U$ on the interval $(0,T_{max})$.

\section{\label{sec:Conservation-of-Momenta}Conservation of Momenta }

Let $U_{0}\in\mathcal{L}$ and $U$ be a weak solution given by Theorem
\ref{thm:ExWeak}. We define 
\[
L(t):=\mathrm{m}\xi(t)+\mathrm{m}_{F}\Omega(t)\times y_{F}
\]
to be the \emph{total linear momentum }of the system, where the second
term is due to the fact that $\Omega$ is calculated with respect
to the center of mass $y_{B}$ of the rigid body and not with respect
to the center of mass of the full structre, $y_{c}$. We denote the
\emph{total angular momentum }of the system by 
\begin{equation}
A(t):=\int_{\mathcal{F}}y\times\bar{u}(t,y)\,\mathrm{d}y+I\Omega(t).\label{eq:defofA}
\end{equation}

\begin{lem}
\label{lem:ConsAngularMomentum}Let $U$ be a weak solution given
by Theorem \ref{thm:ExWeak}. Then
\begin{equation}
L'(t)+\Omega(t)\times L(t)=0,\quad L(0)=\mathrm{m}\xi_{0}+\mathrm{m}_{F}\Omega_{0}\times y_{F}=:L_{0},\label{eq:L}
\end{equation}
and
\begin{equation}
A'(t)+\Omega(t)\times A(t)=0,\quad A(0)=\int_{\mathcal{F}}y\times\bar{u}_{0}(y)\,\mathrm{d}y+I\Omega_{0}=:A_{0}.\label{eq:Aeq}
\end{equation}
In particular, for all $t\geq0$, 
\begin{equation}
\frac{\mathrm{d}}{\mathrm{dt}}\vert A(t)\vert^{2}=\frac{\mathrm{d}}{\mathrm{d}t}\vert L(t)\vert^{2}=0,\qquad\vert A(t)\vert=\vert A_{0}\vert,\vert L(t)\vert=\vert L_{0}\vert.\label{eq:conservedModulus}
\end{equation}
\end{lem}
\begin{proof}
In the weak formulation \eqref{eq:WeakFormula}, for $j\in\{1,2,3\}$,
we take the $j$-th unit vector $e_{j}$ as well as the functions
$y\times e_{j}$ as test functions $\phi$. In \cite[p. 18]{SilvestreTakahashi2012},
it is shown that, by integration by parts and an approximation argument,
it follows that 
\[
\frac{\mathrm{d}}{\mathrm{d}t}(\mathrm{m}_{B}\xi(t)+\int_{\mathcal{F}}u(t)\,\mathrm{d}y)=-\Omega(t)\times(\mathrm{m}_{B}\xi(t)+\int_{\mathcal{F}}u(t)\,\mathrm{d}y)
\]
 and 
\begin{equation}
\frac{\mathrm{d}}{\mathrm{d}t}\left(I_{B}\Omega(t)+\int_{\mathcal{F}}y\times u(t,y)\,\mathrm{d}y\right)=-\Omega(t)\times\left(I_{B}\Omega(t)+\int_{\mathcal{F}}y\times u(t,y)\,\mathrm{d}y\right)-\xi(t)\times\int_{\mathcal{F}}u(t)\,\mathrm{d}y.\label{eq:AeqZwischenschritt}
\end{equation}
Since $\bar{u}(t)\in L_{\sigma}^{2}(\mathcal{F})$ for every weak
solution, $\int_{\mathcal{F}}\bar{u}(t)\,\mathrm{d}y=0$ and \eqref{eq:L}
follows directly. By \eqref{eq:ItoIBIF}, 
\[
I_{B}\Omega(t)+\int_{\mathcal{F}}y\times u(t,y)\,\mathrm{d}y=A(t)+\frac{\mathrm{m}_{F}}{\mathrm{m}}y_{F}\times L(t)
\]
and by \eqref{eq:aTimesb}, 
\[
\Omega\times(A+\frac{\mathrm{m}_{F}}{\mathrm{m}}y_{F}\times L)=\Omega\times A+\frac{\mathrm{m}_{F}}{\mathrm{m}}y_{F}\times(\Omega\times L)-\frac{\mathrm{m}_{F}}{\mathrm{m}}L\times(\Omega\times y_{F}),
\]
so that by \eqref{eq:AeqZwischenschritt} and \eqref{eq:L}, 
\[
\frac{\mathrm{d}}{\mathrm{d}t}A(t)=-\Omega(t)\times A(t)+\mathrm{m}_{F}\xi\times(\Omega\times y_{F})-\xi(t)\times\int_{\mathcal{F}}u(t)\,\mathrm{d}y.
\]
Again since $\bar{u}(t)\in L_{\sigma}^{2}(\mathcal{F})$, $\xi\times\int_{\mathcal{F}}u\,\mathrm{d}y=\xi\times(\Omega\times\mathrm{m}_{F}y_{F})$,
which shows that equation \eqref{eq:Aeq} holds. \eqref{eq:conservedModulus}
follows by multiplying \eqref{eq:L} by $L$ and \eqref{eq:Aeq} by
$A$, respectively.\end{proof}
\begin{rem}
Lemma \ref{lem:ConsAngularMomentum} shows that in the inertial frame,
the total momenta $l(t):=Q(t)L(t)$ and $a(t):=Q(t)A(t)$ are conserved,
i.e. $l(t)=L_{0}$ and $a(t)=A_{0}$ by \eqref{eq:DefDerTrafo}. In
particular, if $y_{F}=x_{F}(0)=x_{B}(0)$, we obtain a constant translational
movement of the system in the inertial frame, $\eta'(t)=0$. 
\end{rem}
Since $\vert A(t)\vert$ is conserved by the system, for the study
of asymptotic behavior of solutions it is convenient to associate
a rigid angular velocity $\io(t):=I^{-1}A(t)$ to $A$ for all $t\geq0$.
It follows that 
\[
\mathrm{I}\io'+\Omega\times\mathrm{I}\io=0
\]
and that 
\[
\vert\mathrm{I}\io(t)\vert=\vert I\io(0)\vert=\vert A_{0}\vert
\]
for all $t\geq0$. We define the rigid angular velocity of the system
relative to $\io$ as 
\[
\tilde{\Omega}:=\Omega-\io,
\]
and note that this implies 
\begin{equation}
\int_{\mathcal{F}}y\times\bar{u}(y)\,\mathrm{d}y=-\mathrm{I}\tilde{\Omega}\label{eq:amTilde}
\end{equation}
by \eqref{eq:defofA}. 

In order to state equations \eqref{eq:FFP transformiert} in terms
of $\bar{u}$ and $\io$, we define a \emph{relative }pressure 
\[
\bar{p}(y):=p(y)+(\xi'+\Omega\times\xi)\cdot y+\frac{1}{2}\vert\Omega\times y\vert^{2}
\]
 and note that the first line in \eqref{eq:FFP transformiert} can
be expressed as 
\begin{equation}
\begin{array}{rcll}
\bar{u}'+\Omega'\times y-\nu\Delta\bar{u}+\nabla\bar{p}+2\Omega\times\bar{u}+(\bar{u}\cdot\nabla)\bar{u} & = & 0, & \text{in }(0,T)\times\mathcal{F},\end{array}\label{eq:A-1}
\end{equation}
where $\bar{p}$ has absorbed all dependence on $\xi$. Moreover,
Lemma \ref{lem:ConsAngularMomentum} shows that given $\Omega$, the
translational velocity $\xi$ can be calculated a posteriori and that
the translational movement of the center of mass can be decoupled
from the remaining system in both the weak and the strong setting.
An equivalent formulation of \eqref{eq:FFP transformiert} in terms
of $\bar{u}$ and $\io$ is thus given by 
\begin{equation}
\left\{ \begin{array}{rcll}
\bar{u}'+\Omega'\times y-\nu\Delta\bar{u}+\nabla\bar{p}+2\Omega\times\bar{u}+(\bar{u}\cdot\nabla)\bar{u} & = & 0, & \text{in }(0,T)\times\mathcal{F},\\
\mathrm{div}\,\bar{u} & = & 0, & \text{in }(0,T)\times\mathcal{F},\\
\bar{u} & = & 0, & \text{on }(0,T)\times\Gamma,\\
I\io'+\Omega\times I\io & = & 0, & t\in(0,T),
\end{array}\right.\label{eq:A-1-1}
\end{equation}
with initial conditions $\bar{u}(0)=\bar{u}_{0}$, $I\io(0)=A_{0}$.
This reduction naturally also shows in the kinetic energy $E(t):=\Vert U(t)\Vert_{L^{2}(\mathcal{S})}^{2}$.
For strong solutions of the full system \eqref{eq:FFP transformiert},
corresponding to \eqref{eq:EI}, we obtain the energy equation
\begin{equation}
E(t)+\int_{s}^{t}\Vert\nabla\bar{u}(\tau)\Vert_{2}^{2}\,\mathrm{d}\tau=E(s)\quad\text{for all }0\leq s<t\leq T_{max}.\label{eq:EE_1}
\end{equation}
Using \eqref{eq:amTilde} and $\int_{\mathcal{F}}\bar{u}(t)\,\mathrm{d}y=0$,
we calculate, both for weak and strong solutions, 
\begin{eqnarray*}
\Vert U(t)\Vert_{L^{2}(\mathcal{S})}^{2} & = & \Vert\bar{u}(t)+\Omega(t)\times\cdot+\xi(t)\Vert_{L^{2}(\mathcal{S})}^{2}\\
 & = & \Vert\bar{u}(t)\Vert_{2}^{2}-2\tom I\tom(t)+2\io I\tom(t)+\Vert\Omega(t)\times\cdot\Vert_{L^{2}(\mathcal{S})}^{2}+\mathrm{m}\vert\xi(t)\vert^{2}+2\xi(t)\cdot(\Omega(t)\times y_{c}).
\end{eqnarray*}
Note that by \eqref{eq:ItoIBIF}, 
\[
\Vert\Omega\times\cdot\Vert_{L^{2}(\mathcal{S})}^{2}=\Omega I\Omega+\mathrm{m}\vert\Omega\times y_{c}\vert^{2},
\]
so $E(t)=\Vert U(t)\Vert^{2}=\Vert\bar{u}(t)\Vert_{2}^{2}-\tom I\tom(t)+\io I\io(t)+\frac{1}{\mathrm{m}}\vert L(t)\vert^{2}$.
By Lemma \ref{lem:ConsAngularMomentum}, $\frac{\mathrm{d}}{\mathrm{d}t}\vert L(t)\vert^{2}=0$,
so that we can ignore this contribution in \eqref{eq:EI} and \eqref{eq:EE_1}
and abuse notation by referring to the kinetic energy 
\begin{equation}
E(t):=\Vert\bar{u}(t)\Vert_{2}^{2}-\tom I\tom(t)+\io I\io(t)\label{eq:Def_of_Energy}
\end{equation}
in the following. 
\begin{rem}
The kinetic energy $E(t)=\bar{E}(t)+\tilde{E}(t)$ splits into a rigid
part 
\[
\bar{E}(t)=\io I\io(t)
\]
and a \emph{positive} {}``relative'' fluid part 
\[
\tilde{E}(t)=\Vert w(t)\Vert_{L^{2}(\mathcal{S})}^{2}=\Vert\bar{u}(t)\Vert_{2}^{2}-\tom I\tom\geq0,
\]
where $w\in L^{2}(\mathcal{S})$ is given by 

\begin{equation}
\begin{cases}
\begin{array}{ll}
\bar{u}(t,y)+\tilde{\Omega}(t)\times y, & y\in\mathcal{F},\\
\tilde{\Omega}(t)\times y, & y\in\mathcal{B}.
\end{array}\end{cases}\label{eq:DefOfw}
\end{equation}
As $I$ is a positive matrix, we also have 
\begin{equation}
\tilde{E}(t)\leq\Vert\bar{u}(t)\Vert_{2}^{2}.\label{eq:EtildeUbar}
\end{equation}

\end{rem}
$\strut$
\begin{rem}
\label{rem: Exponential rate}There are two special cases in which
the rigid part $\bar{E}(t)$ is constant in time. By \eqref{eq:EE_1},
Gronwall's lemma, \eqref{eq:EtildeUbar} and Poincaré's inequality,
both cases imply exponential decay of $\tilde{E}(t)$. The first case
is that of $\vert A_{0}\vert=0$, which implies $\io(t)=I^{-1}A(t)=0$
for all $t\geq0$. This is treated in \cite{SilvestreTakahashi2012}
as the \emph{orthogonality condition}. The second case is the one
of $\mathcal{S}$ essentially being a sphere, i.e. $I=\mathrm{Id}_{\mathbb{R}^{3}}$
and $\io(t)=I^{-1}A_{0}$ for all $t\geq0$ by Lemma \ref{lem:ConsAngularMomentum}.
In this situation, the full structure does not have any preferred
direction of rotation, so that the fluid is only driven by its own
inertia. We briefly consider this special case again in Section \ref{sec:A-priori-characterizations}.
\end{rem}

\section{\label{sec:Properties-of-SDS}Properties of the semidynamical system
$(\bar{u},\io)$}

By Theorem \ref{thm:mainResultN-1}, given $U_{0}\in\mathcal{W}$,
there exists a strong solution $U\in\mathcal{X}^{T_{max}}$, which
implies $\bar{u},\io\in X_{2,2}^{T_{max}}\times H^{1}(\mathbb{R}^{3})$.
From the embedding \eqref{eq:BUCeinbettung} and Lemma \ref{lem:ConsAngularMomentum},
we deduce 
\[
(\bar{u},\io)\in BUC([0,T];H\times\Phi_{\vert A_{0}\vert})\qquad\text{for every }0<T<T_{max},
\]
where we recall $H=H_{0}^{1}(\mathcal{F})\cap L_{\sigma}^{2}(\mathcal{F})$
and we define the ellipsoid 
\[
\Phi_{\vert A_{0}\vert}:=\{r\in\mathbb{R}^{3}:\vert Ir\vert=\vert A_{0}\vert\}.
\]
For the remainder of the proof of Theorem \ref{thm:MR}, we choose
an arbitrary but fixed $A_{0}\in\mathbb{R}^{3}$ and, for every $\delta>0$,
we define 
\[
\mathcal{Z}:=H\times\Phi_{\vert A_{0}\vert},\quad\mathcal{Z}_{\delta}:=\{z=(\bar{v},a)\in\mathcal{Z}:\Vert\bar{v}\Vert_{H}\leq\delta\},
\]
endowed with the metric of $H\times\mathbb{R}^{3}$. For all $\delta>0$,
we define the semiflow $S_{t}^{\delta}:\mathcal{Z}_{\delta}\to\mathcal{Z}$
by $S_{t}^{\delta}(\bar{u}_{0},I^{-1}A_{0})=(\bar{u}(t),\io(t))$
for $0\leq t<T_{max}$. In the following, we write $S_{t}$ if no
confusion is possible, as only the domain, but not the map itself
depends on $\delta$. 
\begin{cor}
\label{cor:Corollary}The family $(S_{t})_{t}$ has the following
properties.
\begin{enumerate}
\item $(S_{t})_{0\leq t<T_{\delta,max}}$ defines a semidynamical system
(cf. \cite[Def. 9.1.1]{CazenaveHaraux1998}) from $\mathcal{Z}_{\delta}$
to $\mathcal{Z}$, i.e. 

\begin{enumerate}
\item \label{enu:61a}for all $0\leq t\leq T_{\delta,max}$, $S_{t}\in C(\mathcal{Z}_{\delta},\mathcal{Z})$,
\item \label{enu:61b}$S_{0}=\mathrm{Id}_{\mathcal{Z}}$,
\item \label{enu:61c}for all $0\leq t+s\leq T_{\delta,max}$, $S_{t+s}=S_{t}\circ S_{s}$,
\item \label{enu:61d}the function $t\mapsto S_{t}z$ is in $C([0,T_{\delta,max});\mathcal{Z})$
for all $z\in\mathcal{Z}_{\delta}$.
\end{enumerate}
\item \label{enu:62}For every $\delta>0$, there is a time $0<T_{\delta}\leq T_{max}$
such that for all $z\in\mathcal{Z}_{\delta}$ and for all $0\leq t\leq T_{\delta}$,
\[
\Vert\bar{u}(t)\Vert_{H}\leq2\delta.
\]

\item \label{enu:63}Given $\vert A_{0}\vert,\delta>0$, we can choose an
interval of existence of strong solutions uniformly in $\mathcal{Z}_{\delta}$,
i.e. $0<T_{\delta,max}:=\inf_{(\bar{u}_{0},I^{-1}A_{0})\in\mathcal{Z}_{\delta}}T_{max}(\bar{u}_{0},I^{-1}A_{0})$
exists.
\end{enumerate}
\end{cor}
\begin{proof}
\eqref{enu:61a} follows from continuous dependence of strong solutions
on their data, where $T_{\delta,max}$ is chosen as in \eqref{enu:63}.
Given two initial data $z_{0}^{1},z_{0}^{2}\in\mathcal{Z}_{\delta}$
and corresponding solutions $U^{1}$ and $U^{2}$, by \eqref{eq:BUCeinbettung},
\[
\sup_{t\in(0,T_{\delta,max})}\Vert S_{t}z_{0}^{1}-S_{t}z_{0}^{2}\Vert_{\mathcal{Z}}\leq C(T_{\delta,max})\Vert U^{1}-U^{2}\Vert_{\mathcal{X}_{2,2}^{T_{\delta,max}}}\to0\;\text{as }z_{0}^{1}\to z_{0}^{2}
\]
by Theorem \ref{thm:mainResultN-1}. Thus, for all $0\leq t\leq T_{\delta,max}$,
$S_{t}z_{0}^{1}\to S_{t}z_{0}^{2}$ in $\mathcal{Z}$ as $z_{0}^{1}\to z_{0}^{2}$
in $\mathcal{Z}$. \eqref{enu:61d}, \eqref{enu:61b} follow directly
from \eqref{eq:BUCeinbettung} and \eqref{enu:61c} follows by definition. 

In order to prove \eqref{enu:62}, we multiply the first line in \eqref{eq:A-1-1}
by $\bar{u}'$ and integrate over $\mathcal{F}$. Using integration
by parts on $-\nu\int_{\mathcal{F}}\Delta\bar{u}\cdot\bar{u}'\,\mathrm{d}y$,
which can be justified by approximation in the strong setting, we
obtain:
\begin{equation}
\frac{\nu}{2}\frac{\mathrm{d}}{\mathrm{d}t}\Vert\nabla\bar{u}(t)\Vert_{2}^{2}+\Vert\bar{u}'(t)\Vert_{2}^{2}=\Omega'\cdot I\tom'(t)-\int_{\mathcal{F}}(((\bar{u}\cdot\nabla)\bar{u})\cdot\bar{u}')(t)\,\mathrm{d}y-2\int_{\mathcal{F}}((\Omega\times\bar{u})\cdot\bar{u}')(t)\,\mathrm{d}y.\label{eq:HOE1}
\end{equation}
The second term on the right-hand side satisfies $\int_{\mathcal{F}}((\bar{u}\cdot\nabla)\bar{u})\cdot\bar{u}'\,\mathrm{d}y\leq\frac{\mu}{4}\Vert\bar{u}'\Vert_{2}^{2}+C\Vert(\bar{u}\cdot\nabla)\bar{u}\Vert_{2}^{2}$,
where $\mu>0$ has to be a small constant which will be determined
below in \eqref{eq:ChoiceOfMu}. We obtain 
\begin{eqnarray*}
\Vert((\bar{u}\cdot\nabla)\bar{u})(t)\Vert_{2}^{2} & \leq & \Vert\bar{u}(t)\Vert_{6}^{2}\Vert\nabla\bar{u}(t)\Vert_{3}^{2}\\
 & \leq & C\Vert\bar{u}(t)\Vert_{H^{1}(\mathcal{F})}^{2}\Vert\nabla\bar{u}(t)\Vert_{2}\Vert\nabla\bar{u}(t)\Vert_{6}\\
 & \leq & C\Vert\nabla\bar{u}(t)\Vert_{2}^{3}(\Vert\nabla\bar{u}(t)\Vert_{2}+\Vert D^{2}\bar{u}(t)\Vert_{2}),
\end{eqnarray*}
by Hölder's inequality, the Sobolev embedding $H^{1}(\mathcal{F})\hookrightarrow L^{6}(\mathcal{F})$,
interpolation for $L^{3}(\mathcal{F})$ and Poincar\'{e}'s inequality.
Moreover, $\bar{u}(t)$ satisfies the stationary Stokes problem
\[
\left\{ \begin{array}{rcll}
-\mu\Delta\bar{u}+\nabla\bar{p} & = & -\bar{u}'-\Omega'\times y-2\Omega\times\bar{u}-(\bar{u}\cdot\nabla)\bar{u}, & \text{in }\mathcal{F},\\
\mathrm{div}\,\bar{u} & = & 0, & \text{in }\mathcal{F},\\
\bar{u} & = & 0, & \text{on }\Gamma,
\end{array}\right.
\]
almost everywhere in time. Thus, by properties of the Stokes operator
(cf. e.g. \cite{Solonnikov77}), 
\[
\Vert D^{2}\bar{u}\Vert_{2}\leq C\left(\Vert\bar{u}'\Vert_{2}+\Vert\Omega'\times y\Vert_{2}+\Vert\Omega\times\bar{u}\Vert_{2}+\Vert(\bar{u}\cdot\nabla)\bar{u}\Vert_{2}+\Vert\nabla\bar{u}\Vert_{2}\right).
\]
It follows that
\begin{equation}
\Vert(\bar{u}\cdot\nabla)\bar{u}\Vert_{2}^{2}\leq C\Vert\nabla\bar{u}\Vert_{2}^{3}\left(\Vert\bar{u}'\Vert_{2}+\vert\tom'\vert^{2}+\vert\io'\vert^{2}+\vert\Omega\vert\Vert\bar{u}\Vert_{2}+\Vert(\bar{u}\cdot\nabla)\bar{u}\Vert_{2}+\Vert\nabla\bar{u}\Vert_{2}\right).\label{eq:sinnlos2}
\end{equation}
Note that $\vert\tom\vert\leq C\Vert\bar{u}\Vert_{2}$ , $\vert\io\vert\leq C\vert A_{0}\vert$
and 
\[
\io'I\io'\leq C(\vert A_{0}\vert^{2}+\Vert\bar{u}\Vert_{2}^{2})
\]
by \eqref{eq:Aeq}. Thus, by Young's and Poincaré's inequalities,
\eqref{eq:sinnlos2} implies 
\[
\frac{1}{2}\Vert(\bar{u}\cdot\nabla)\bar{u}\Vert_{2}^{2}\leq C\Vert\nabla\bar{u}\Vert_{2}^{6}+\frac{\mu}{4}\Vert\bar{u}'\Vert_{2}^{2}+C\Vert\nabla\bar{u}\Vert^{3}\left(\vert A_{0}\vert^{2}+(1+\vert A_{0}\vert)\Vert\nabla\bar{u}\Vert_{2}+\Vert\nabla\bar{u}\Vert_{2}^{2}\right).
\]
In conclusion, the second term on the right-hand side of \eqref{eq:HOE1}
satisfies 
\[
\vert\int_{\mathcal{F}}((\bar{u}\cdot\nabla)\bar{u})\cdot\bar{u}'\,\mathrm{d}y\vert\leq\frac{\mu}{2}\Vert\bar{u}'\Vert_{2}^{2}+C\Vert\nabla\bar{u}\Vert_{2}^{3}\left(\vert A_{0}\vert^{2}+(1+\vert A_{0}\vert)\Vert\nabla\bar{u}\Vert_{2}+\Vert\nabla\bar{u}\Vert_{2}^{2}+\Vert\nabla\bar{u}\Vert_{2}^{3}\right).
\]
The last term on the right-hand side of \eqref{eq:HOE1} satisfies
\[
\vert\int_{\mathcal{F}}(\Omega\times\bar{u})\cdot\bar{u}'\,\mathrm{d}y\vert\leq C(\vert A_{0}\vert+\Vert\bar{u}\Vert_{2}^{2})\Vert\bar{u}\Vert_{2}^{2}+\frac{\mu}{4}\Vert\bar{u}'\Vert_{2}^{2}
\]
and for the first term on the right-hand side of \eqref{eq:HOE1},
we obtain
\begin{eqnarray*}
\io'I\tom' & \leq & C\io'I\io'+\frac{\mu}{4}\Vert\bar{u}'\Vert_{2}^{2}\\
 & \leq & C(\vert A_{0}\vert^{2}+\Vert\bar{u}\Vert_{2}^{2})+\frac{\mu}{4}\Vert\bar{u}'\Vert_{2}^{2},
\end{eqnarray*}
so that \eqref{eq:HOE1} becomes 
\begin{eqnarray}
 &  & \frac{1}{2}\frac{\mathrm{d}}{\mathrm{d}t}\Vert\nabla\bar{u}\Vert_{2}^{2}+(1-\mu)\Vert\bar{u}'\Vert_{2}^{2}-\tom'I\tom'\label{eq:sinnlos 3}\\
 & \leq & C\vert A_{0}\vert^{2}+C\Vert\nabla\bar{u}\Vert_{2}^{2}\left(1+\vert A_{0}\vert+\vert A_{0}\vert^{2}\Vert\nabla\bar{u}\Vert_{2}+(1+\vert A_{0}\vert)\Vert\nabla\bar{u}\Vert_{2}^{2}+\Vert\nabla\bar{u}\Vert_{2}^{3}+\Vert\nabla\bar{u}\Vert_{2}^{4}\right).\nonumber 
\end{eqnarray}
Note that by \eqref{eq:DefOfw}, we obtain 
\begin{equation}
(1-\mu)\Vert\bar{u}'\Vert_{2}^{2}-\tom'I\tom'=(1-\mu)\Vert w'\Vert_{2}^{2}+(1-2\mu)\tom'I_{B}\tom'-\mu\tom'I_{F}\tom'\geq0,\label{eq:ChoiceOfMu}
\end{equation}
if $\mu$ is chosen sufficiently small, depending on the {}``ratio''
of $I_{B}$ and $I_{F}$. Integrating \eqref{eq:sinnlos 3} in time
yields
\begin{equation}
\Vert\nabla\bar{u}\Vert_{2}^{2}(t)\leq\Vert\nabla\bar{u}_{0}\Vert_{2}^{2}+Ct\vert A_{0}\vert^{2}+\sum_{i=2}^{6}C(i,\vert A_{0}\vert)\int_{0}^{t}\Vert\nabla\bar{u}\Vert_{2}^{i}(s)\,\mathrm{d}s.\label{eq:HOE}
\end{equation}
Now it is clear that given $\Vert\bar{u}_{0}\Vert_{H}\leq\delta$,
$\Vert\bar{u}(t)\Vert_{H}^{2}\leq4\delta^{2}$ holds as long as $t\leq T_{\delta}:=\frac{C\delta^{2}}{\vert A_{0}\vert^{2}+\sum_{i=2}^{6}\delta^{i}},$
where $C$ is a constant depending on $\mathcal{B}$, $\mathcal{F}$
and $\rho_{B}$. Moreover, since the moduli of $A$ and $L$ are conserved
along solutions, the blow-up criterion given in \eqref{eq:BlowUp}
reduces to 
\begin{equation}
\Vert\bar{u}(t)\Vert_{H}\to\infty\text{ for }t\to T_{max}\label{eq:trueBlowUp}
\end{equation}
and thus estimate \eqref{eq:HOE} shows \eqref{enu:63}.
\end{proof}

\section{\label{sec:Properties-of-WS}Properties of the weak solution}

There may be several different methods of constructing weak solutions
for \eqref{eq:FFP transformiert}, but of course, we do not show uniqueness
of global solutions here, so let us state the requirements needed
of weak solutions in general in order to make our subsequent arguments
work. We then show that the solutions constructed in \cite{SilvestreTakahashi2012}
satisfy these requirements.
\begin{assumption}
\label{Assumption on WS}Given $U_{0}\in\mathcal{L}$, there is a
weak solution $U\in C_{w}([0,\infty;\mathcal{L})\cap L_{loc}^{2}(0,\infty;\mathcal{H})$
of \eqref{eq:FFP transformiert} which satisfies
\begin{enumerate}
\item \label{enu:711}for all $t\geq0$, $\vert I\io\vert^{2}(t)=\vert A_{0}\vert^{2},$
\item \label{enu:712}the \emph{strong energy inequality}, i.e. for almost
all $0\leq s<t\leq\infty$, 
\begin{equation}
E(t)+2\nu\int_{s}^{t}\Vert\nabla\bar{u}(\tau)\Vert_{2}^{2}\,\mathrm{d}\tau\leq E(s),\label{eq:SEI}
\end{equation}
where $E(s)$ is defined in \eqref{eq:Def_of_Energy}.
\item \emph{\label{enu:713}weak-strong-uniqueness}, i.e. if $U_{0}\in\mathcal{W}$,
then $U$ is unique on $[0,T_{max}(U_{0}))$ and it is equal to the
strong solution $W\in\mathcal{X}^{T_{max}}$, also emenating from
$U_{0}$, given on $[0,T_{max})$ by Theorem \ref{thm:mainResultN-1}.
In particular, every strong solution is a weak solution.
\end{enumerate}
\end{assumption}
\begin{cor}
The weak solution $U$ given in Theorem \ref{thm:ExWeak} satisfies
Assumption \ref{Assumption on WS}. \end{cor}
\begin{proof}
Property \eqref{enu:711} was proved in Lemma \ref{lem:ConsAngularMomentum}.
Since $\nabla\bar{u}=D(U)$ on $\mathcal{F}$ , for $s=0$, \eqref{enu:712}
is a consequence of \eqref{eq:EI} and the discussion of \eqref{eq:Def_of_Energy}
in Section \ref{sec:Conservation-of-Momenta} and we note that for
a general weak solution, \eqref{enu:712} in this form implicitly
gives conservation of linear momentum. In \eqref{eq:EI}, $U$ is
constructed by a Galerkin approximation where the approximants $U_{k}$
satisfy the energy equality \eqref{eq:EE_1}, and they converge strongly
to $U$ in the norm $L_{loc}^{2}(0,\infty;\mathcal{L})$ for a subsequence
\cite[p. 16]{SilvestreTakahashi2012}. This implies $\Vert U_{k}(s)\Vert_{2}^{2}\to\Vert U(s)\Vert_{2}^{2}$
for almost all $s\in(0,\infty)$ for a subsequence, so that \eqref{eq:SEI}
also follows by passing to the limit and weak lower semicontinuity
of the norm. 

In order to prove \eqref{enu:713}, we apply $W$ as a test function
for $U$ and obtain
\begin{equation}
\langle U',W\rangle+a(U,W)+b(U,U,W)=0,\qquad\text{a.e. in }(0,T_{max}).\label{eq:UtestW}
\end{equation}
At the same time, we can apply the approximants $U_{k}$ of $U$ as
test functions for $W$ to get 
\begin{equation}
\langle W',U_{k}\rangle+a(W,U_{k})+b(W,W,U_{k})=0,\qquad\text{a.e. in }(0,T_{max}).\label{eq:WtestV}
\end{equation}
We integrate both equations in time and note that 
\[
\int_{0}^{t}\langle W',U_{k}\rangle(s)\,\mathrm{d}s=-\int_{0}^{t}\langle W,U_{k}'\rangle(s)\,\mathrm{d}s+\langle W,U\rangle(t)-\Vert U_{0}\Vert_{L^{2}(\mathcal{S})}^{2}.
\]
We add up the (original version of the) energy inequality for $U$,
\[
\Vert U(t)\Vert_{L^{2}(\mathcal{S})}^{2}+2\int_{0}^{t}a(U,U)(s)\,\mathrm{d}s\leq\Vert U_{0}\Vert_{L^{2}(\mathcal{S})}^{2}
\]
and the energy equality for $W$, 
\[
\Vert W(t)\Vert_{L^{2}(\mathcal{S})}^{2}+2\int_{0}^{t}a(W,W)(s)\,\mathrm{d}s=\Vert U_{0}\Vert_{L^{2}(\mathcal{S})}^{2}
\]
and subtract twice the time integrals of \eqref{eq:WtestV} and \eqref{eq:UtestW}
and pass to the limit in the linear terms to obtain 
\begin{equation}
\Vert(U-W)(t)\Vert_{L^{2}(\mathcal{S})}^{2}+2\int_{0}^{t}a(U-W,U-W)\,\mathrm{d}s\leq\int_{0}^{t}b(W,W,U_{k})(s)+b(U,U,W)(s)\,\mathrm{d}s.\label{eq:some w-s-estimate}
\end{equation}
Manipulation and passage to the limit on the right-hand side is critical
only in the terms of type $\int_{0}^{t}\int_{\mathcal{F}}((\bar{w}\cdot\nabla)\bar{w})\cdot\bar{u}_{k}\,\mathrm{d}y\mathrm{d}t$,
which work as usual for the Navier-Stokes problem and are justified
also in \cite[p. 13]{SilvestreTakahashi2012}. In particular, we can
show here that 
\[
\lim_{k\to\infty}\int_{0}^{t}b(W,W,U_{k})(s)+b(U,U,W)(s)\,\mathrm{d}s=\int_{0}^{t}b(U-W,U-W,W)(s)\,\mathrm{d}s,
\]
using $b(W,W,U_{k})=-b(U_{k},W,W)-\Omega_{W}\cdot(\Omega_{U}\times I_{B}\Omega_{W})$
and $b(W,U,W)=\Omega_{W}\cdot(\Omega_{U}\times I_{B}\Omega_{W})$.
Let $Z:=U-W$. It remains to estimate 
\begin{eqnarray}
\int_{0}^{t}b(Z,Z,W)\,\mathrm{d}s & \leq & C\int_{0}^{t}\Vert W\Vert_{L^{\infty}(\mathcal{S})}(\Vert z\Vert_{2}^{2}+\Omega_{Z}I_{B}\Omega_{Z}+m_{B}\vert\xi_{Z}\vert^{2})\,\mathrm{d}s\nonumber \\
 &  & +\nu\int_{0}^{t}\Vert\nabla\bar{z}\Vert_{2}^{2}\,\mathrm{d}s+C\int_{0}^{T}\Vert w\Vert_{\infty}^{2}\Vert z\Vert_{2}^{2}\,\mathrm{d}s\nonumber \\
 & \leq & \int_{0}^{t}a(Z,Z)\,\mathrm{d}s+C\int_{0}^{t}\Vert W\Vert_{L^{\infty}(\mathcal{S})}^{2}\Vert Z\Vert_{L^{2}(\mathcal{S})}^{2}(s)\,\mathrm{d}s\label{eq: some w-s-estimate 2}
\end{eqnarray}
by Hölder's and Young's inequalities. Since $s\mapsto\Vert W(s)\Vert_{L^{\infty}(\mathcal{S})}^{2}$
is integrable on $(0,T),$ $T<T_{max}$ by the assumption $W\in\mathcal{X}_{2,2}^{T}$,
by Gronwall's Lemma, \eqref{eq:some w-s-estimate} and \eqref{eq: some w-s-estimate 2}
imply $U=W$ on $(0,T)$. 
\end{proof}

\section{\label{sec:Strong-solutions-for-large-time}Strong solutions for
large time}

For the Navier-Stokes equations, the strong energy inequality and
weak-strong uniqueness imply a Leray Structure Theorem (\cite{Leray1934b}
and cf. \cite[Sect. 6]{GaldiLecture2000} for a survey on known results
depending on the fluid domain). In particular, every weak solution
can be shown to remain regular after some (possibly large) time.
\begin{prop}
\label{prop:GlobalSol}Given $A_{0}\in\mathbb{R}^{3}$ and $\bar{u}_{0}\in L_{\sigma}^{2}(\mathcal{F})$,
for every weak solution $U$ satisfying Assumption \ref{Assumption on WS}, 
\begin{enumerate}
\item there is a time $T_{*}(\bar{u}_{0},A_{0})$ such that $(\bar{u},\io)(T_{*})\in H\times\mathbb{R}^{3}$
and $(\bar{u},\io)$ is the unique strong solution of \eqref{eq:A-1-1}
on $(T_{*},\infty)$ with initial values $(\bar{u},\io)(T_{*})$. 
\item $\Vert\bar{u}(t)\Vert_{H}\to0$ as $t\to\infty$. 
\end{enumerate}
\end{prop}
\begin{proof}
By assumption, we have \eqref{eq:SEI} and thus the total dissipation
\begin{equation}
2\nu\int_{0}^{\infty}\Vert\nabla\bar{u}(\tau)\Vert_{2}^{2}\,\mathrm{d}\tau\leq E(0)\label{eq: bounded dissipation}
\end{equation}
 is bounded. It follows that for every $d>0$, there is a time $t_{d}\geq0$
such that 
\[
\int_{t_{d}}^{\infty}\Vert\nabla\bar{u}(s)\Vert_{2}^{2}\,\mathrm{d}s<d.
\]
Let $\delta>0$ and $d\leq\frac{T_{\delta}}{2}\delta(1+C_{p})$, where
$T_{\delta}$ is the constant from Corollary \ref{cor:Corollary}.
Again by \eqref{eq: bounded dissipation}, there is a time $T_{*}(\delta)\geq t_{d}$
such that $\Vert\bar{u}(T_{*})\Vert_{H}\leq\delta$ and this choice
is reasonable since $\bar{u}\in C_{w}([0,\infty;L^{2}(\mathcal{F}))$.
By Theorem \ref{thm:mainResultN-1} and Assumption \ref{Assumption on WS}.3,
$(\bar{u},\io)$ is unique and strong on $(T_{*},T_{*}+T_{\delta})$.
Moreover, there must be a point in time $t_{\delta}\in(T_{*}+\frac{T_{\delta}}{2},T_{*}+T_{\delta})$,
such that again $\Vert\bar{u}(t_{\delta})\Vert_{H}\leq\delta$, because
if we assume the contrary, then 
\[
\int_{T_{*}+\frac{T_{\delta}}{2}}^{T_{*}+T_{\delta}}\Vert\bar{u}(s)\Vert_{H}\,\mathrm{d}s\geq(1+\frac{1}{C_{p}})\int_{T_{*}+\frac{T_{\delta}}{2}}^{T_{*}+T_{\delta}}\Vert\nabla\bar{u}(s)\Vert_{2}\,\mathrm{d}s\geq(1+\frac{1}{C_{p}})\frac{T_{\delta}}{2}\delta>d.
\]
This shows that for every $\delta>0$, the strong solution starting
from $T_{*}$ can be extended indefinitely by glueing together intervals
of length $>\frac{T_{\delta}}{2}$. Moreover, for all $t\geq T_{*}(\delta)$,
$\Vert\bar{u}(t)\Vert_{H}\leq2\delta$ by Corollary \ref{cor:Corollary}.
This shows that $\Vert\bar{u}(t)\Vert_{H}\to0$ as $t\to\infty$. 
\end{proof}

\section{\label{sec:Application-of-LaSalle's}Application of LaSalle's invariance
principle and proof of the main result}

It remains to show the asymptotics for $\Omega$ in Theorem \ref{thm:MR}.
We use a modification of LaSalle's Invariance Principle. Following
along the lines of the proof of \cite[Thm. 9.2.7]{CazenaveHaraux1998},
we reprove this principle in the present situation as small adjustments
have to be made due to the facts that we cannot a priori work with
a globally defined semiflow on a complete metric space and that we
have to single out trajectories. Throughout this section, let $A_{0}\in\mathbb{R}^{3}$
and $\bar{u}_{0}\in L_{\sigma}^{2}(\mathcal{F})$ be fixed. Let $U$
be a weak solution corresponding to these inital values and choose
$\delta>0$, e.g. $\delta=1$, such that $T_{*}=T_{*}(\delta)$ from
the proof of Proposition \ref{prop:GlobalSol} is given with $\Vert\bar{u}(t)\Vert_{H}\leq2\delta$
for all $t\geq T_{*}$. We set 
\[
z_{0}:=(\bar{u}(T_{*}),\io(T_{*}))
\]
and start a new time $t\in\mathbb{R}_{+}$ at $T_{*}$. We define
$\mathcal{O}(z_{0}):=\bigcup_{t\geq0}\{S_{t}z_{0}\}$ to be the orbit
of $z_{0}$ and note that for all $t\geq0$, $S_{t}:\mathcal{O}(z_{0})\subset\mathcal{Z}_{2\delta}\to\mathcal{O}(z_{0})$
is well-defined. Let 
\[
\omega(z_{0}):=\{z\in\mathcal{Z}:\exists t_{n}\overset{n\to\infty}{\to}\infty\:\text{such that }S_{t_{n}}z_{0}\to z\}
\]
be the $\omega$-limit set of $z_{0}$.
\begin{prop}
\label{prop:Lyapunov1}We collect the following properties of $\mathcal{O}(z_{0})$.
\begin{enumerate}
\item \label{enu:911}The closure of $\mathcal{O}(z_{0})$ in $\mathcal{Z}$
satisfies $\overline{\mathcal{O}(z_{0})}^{\mathcal{Z}}=\mathcal{O}(z_{0})\cup\omega(z_{0})\subset\mathcal{Z}_{2\delta}$. 
\item \label{enu:912}$\mathcal{O}(z_{0})$ is relatively compact in $\mathcal{Z}$. 
\item \label{enu:913}We have $\lim_{t\to\infty}d(S_{t}z_{0},\omega(z_{0}))=0$,
where $d(\cdot,\cdot)$ is the distance induced by the $\mathcal{Z}$-norm
on $\mathcal{Z}_{2\delta}$.
\item \label{enu:914}For $0\leq t\leq T_{2\delta,max}$, $\omega(z_{0})$
is invariant under $S_{t}$. 
\end{enumerate}
\end{prop}
\begin{proof}
The first property \eqref{enu:911} follows by defnition. By Proposition
\ref{prop:GlobalSol}, 
\[
\lim_{t\to\infty}S_{t}z_{0}\subset\{0\}\times\Phi_{\vert A_{0}\vert},
\]
which implies \eqref{enu:912}. In order to show \eqref{enu:913},
we assume that to the contrary, there is an $\varepsilon>0$ and a
sequencce $t_{n}\overset{n\to\infty}{\to}\infty$ such that $d(S_{t_{n}}z_{0};\omega(z_{0}))\geq\varepsilon$.
Then by relative compactness, there is a subsequence $t_{n_{k}}\overset{k\to\infty}{\to}\infty$
such that $S_{t_{n_{k}}}z_{0}\to\overline{z}\in\omega(z_{0})$, yielding
a contradiction. Finally, for all $\overline{z}\in\omega(z_{0})\subset Z_{2\delta}$,
\[
S_{t}\overline{z}=S_{t}(\lim_{n\to\infty}S_{t_{n}}z_{0})=\lim_{n\to\infty}S_{t+t_{n}}z_{0}\subset\omega(z_{0})\subset\mathcal{Z}_{2\delta}
\]
 by the continuity of $S_{t}$ on $\mathcal{Z}_{2\delta}$, cf. Corollary
\ref{cor:Corollary}. This proves \eqref{enu:914}.\end{proof}
\begin{prop}
The total kinetic energy $E(\bar{u}(t),\io(t))=E(t)$ is a strict
Lyapunov functional for $\mathcal{O}(z_{0})$ and the equilibrium
set $\mathcal{E}:=\{z\in\mathcal{Z}_{2\delta}:\exists0\leq t<T_{2\delta,max},S_{t}z=z\}$
is characterized by 
\[
\mathcal{E}=\{0\}\times\{\oO\in\Phi_{\vert A_{0}\vert}:\oO\text{ is an eigenvector of }I\}.
\]
\end{prop}
\begin{proof}
The function $E:\mathcal{Z}\to\mathbb{R}_{+}$ is continuous by definition
and decreasing along the trajectory of $z_{0}$ by the energy equality
\eqref{eq:EE_1}. If we assume that for some $z\in\mathcal{Z}_{2\delta}$
and $0<t\leq T_{2\delta,max}$ we have $E(S_{t}z)=E(z)$, then by
\eqref{eq:EE_1}, $\int_{0}^{t}\Vert\nabla\bar{u}\Vert_{2}^{2}(s)\,\mathrm{d}s=0$
and thus $\bar{u}(s)\equiv0$ on $(0,t)$ by Poincaré's inequality.
It follows that $\tom=0$ and $\Omega=\io$ on $(0,t)$. Since $S_{t}z$
gives a strong solution of \eqref{eq:A-1-1}, we obtain 
\[
\io'(s)\times y+\nabla\bar{p}(s,y)=0\;\text{for a.e. }(s,y)\in(0,t)\times\mathcal{F}
\]
from the first line. But the linear function $\io'\times\cdot$ cannot
be the gradient of a function $\bar{p}\in H^{1}(\mathcal{F})$, except
if $\io'=0$.  It follows that $S_{s}z=z$ for all $s\in[0,t]$,
i.e. $z\in\mathcal{E}$ and thus, $E$ is a strict Lyapunov functional.
Line 4 in \eqref{eq:A-1-1} shows that in this situation, $\io\times I\io=0$,
so that the vector $\io$ constant on $(0,t)$ must be an eigenvector
of $I$. This proves the claim on $\mathcal{E}$. 
\end{proof}
It remains to show that $\omega(z_{0})\subset\mathcal{E}$. Since
for all $t_{n}\overset{n\to\infty}{\to}\infty$, $(E(S_{t_{n}}z_{0}))_{n\in\mathbb{N}}$
is monotone and bounded from below, 
\[
E_{\infty}:=\lim_{t\to\infty}E(S_{t}z_{0})
\]
exists and $\text{for all }\bar{z}\in\omega(z_{0})$, 
\begin{equation}
E(\bar{z})=E_{\infty}.\label{eq:EconstOnOmLimesMenge}
\end{equation}
Let $\mathcal{E}_{\infty}:=\{(0,\oO)\in\mathcal{E}:\oO I\oO=E_{\infty})$.
Since $E$ is constant on $\omega(z_{0})$ by \eqref{eq:EconstOnOmLimesMenge}
and $E$ is a strict Lyapunov functional, $\omega(z_{0})\subset\mathcal{E}_{\infty}$.
By Proposition \ref{prop:Lyapunov1}, we conclude that $\lim_{t\to\infty}d(S_{t}z_{0},\mathcal{E}_{\infty})=0$.
Clearly, if the eigenvalues $\lambda_{j},j\in\{1,2,3\}$ of $I$ are
distinct, then $\mathcal{E}_{\infty}$ contains six isolated vectors
in $\mathbb{R}^{3}$ and thus $S_{t}z_{0}$ must converge to one of
them as $t\to\infty$. This proves Theorem \ref{thm:MR}.

\section{\label{sec:A-priori-characterizations}A priori characterizations
of $\oO$ }

From the initial data, we extract some information about which vector
$\oO$ is finally attained, using the elementary relations of $A_{0},E_{0}$
and $E_{\infty}$. Note that if $\vert A_{0}\vert=0$, then always
$\oO=0$, cf. Remark \ref{rem: Exponential rate}. Without loss of
generality, we assume that $I$ is given by a diagonal matrix. 
\begin{enumerate}
\item The first case is $I=\lambda\mathrm{Id}_{\mathbb{R}^{3}}$ for some
$\lambda>0$. This makes $\mathcal{S}$ essentially a sphere but $\mathcal{B}$
and $\mathcal{F}$ individually may still have a much more complicated
geometries. By Lemma \ref{lem:ConsAngularMomentum},
\[
\io(t)=\io(0)=A_{0}\ \text{for all }t\geq0
\]
and by Theorem \ref{thm:MR}, $\Omega(t)\to A_{0}$ as $t\to\infty$.
By Remark \ref{rem: Exponential rate}, the rate of convergence of
$\bar{u}$ in the $L^{2}(\mathcal{F})$-norm is exponetial.
\item The second case is that of $\mathcal{S}$ essentially being an \emph{egg},
i.e. $I=\mathrm{diag}(\lambda_{s},\lambda_{s},\lambda_{l})$ where
$0<\lambda_{s}<\lambda_{l}$. We show the following.

\begin{prop}
If
\begin{equation}
\tilde{E}(0)<\lambda_{l}(\frac{\lambda_{l}}{\lambda_{s}}-1)(\io_{0})_{3}^{2},\label{eq:egg}
\end{equation}
then $\oO=\mu_{l}e_{3}$, where $\mu_{l}$ is determined by $\mu_{l}^{2}\lambda_{l}^{2}=\vert I\oO\vert^{2}=\vert A_{0}\vert^{2}$. \end{prop}
\begin{proof}
We use a contradiction argument and assume that $\oO=\mu_{1}e_{1}+\mu_{2}e_{2}$
for some $\mu_{1},\mu_{2}\in\mathbb{R}$. We show that this can only
occur if the energy initially stored in the $e_{3}$-axis is smaller
than the initial kinetic energy provided by the fluid, which is the
interpretation of \eqref{eq:egg}. The initial kinetic energy for
our problem is given by 
\[
E(0)=\lambda_{s}[(\io_{0})_{1}^{2}+(\io_{0})_{2}^{2}]+\lambda_{l}(\io_{0})_{3}^{2}+\tilde{E}(0)
\]
and the initial absolute value of angular momentum is given by 
\[
\vert A_{0}\vert^{2}=\lambda_{s}^{2}[(\io_{0})_{1}^{2}+(\io_{0})_{2}^{2}]+\lambda_{l}^{2}(\io_{0})_{3}^{2}.
\]
Thus,
\[
E(0)=\frac{\vert A_{0}\vert^{2}}{\lambda_{s}}-\lambda_{l}(\frac{\lambda_{l}}{\lambda_{s}}-1)(\io_{0})_{3}^{2}+\tilde{E}(0).
\]
The final absolute value of the angular momentum is equal to the initial
one and in this case it is given by 
\[
\vert A_{0}\vert^{2}=\vert I\oO\vert^{2}=\lambda_{s}^{2}[(\oO)_{1}^{2}+(\oO)_{2}^{2}].
\]
The final kinetic energy is thus given by 
\[
E_{\infty}=\lambda_{s}[(\oO)_{1}^{2}+(\oO)_{2}^{2}]=\frac{\vert A_{0}\vert^{2}}{\lambda_{s}}.
\]
For all weak solutions satisfying Assumption \ref{Assumption on WS},
by the energy inequality, we may crudely estimate 
\begin{equation}
E_{\infty}\leq E(0),\label{eq:CrudeEE}
\end{equation}
which, by the above calculations, implies 
\[
\lambda_{l}(\frac{\lambda_{l}}{\lambda_{s}}-1)(\io_{0})_{3}^{2}\leq\tilde{E}(0).
\]

\end{proof}
\item Analogous arguments apply in the general case $I=\mathrm{diag}(\lambda_{s},\lambda_{m},\lambda_{l})$,
where $0<\lambda_{s}<\lambda_{m}<\lambda_{l}$. We obtain the following
two characterizations: If initially 
\[
\lambda_{l}(\frac{\lambda_{l}}{\lambda_{m}}-1)(\io_{0})_{3}^{2}>\tilde{E}(0)+\lambda_{s}(1-\frac{\lambda_{s}}{\lambda_{m}})(\io_{0})_{1}^{2},
\]
and 
\begin{equation}
\lambda_{m}(\frac{\lambda_{m}}{\lambda_{s}}-1)(\io_{0})_{2}^{2}+\lambda_{l}(\frac{\lambda_{l}}{\lambda_{s}}-1)(\io_{0})_{3}^{2}>\tilde{E}(0),\label{eq:generalInstability}
\end{equation}
then $\oO=\mu_{l}e_{3}$ with $\mu_{l}^{2}\lambda_{l}^{2}=\vert A_{0}\vert$.
If only \eqref{eq:generalInstability} holds, then still $\oO=\mu_{s}e_{1}$
cannot be attained for any $\mu_{s}\in\mathbb{R}$. \end{enumerate}
\begin{rem}
A priori information about the size of the dissipation $2\nu\int_{0}^{\infty}\Vert\nabla\bar{u}(\tau)\Vert_{2}^{2}\,\mathrm{d}\tau$
improves the estimate in \eqref{eq:CrudeEE} and would yield better
criteria. This information may not be available in general for weak
solutions. It is shown in \cite{LyashenkoFriedlander1998B} that the
viscosity parameter $\nu$ is decisive for the asymptotics and this
can also be seen in numerical simulations \cite{GaldiMazzoneZunino2013}.
\end{rem}
\strut
\begin{rem}
In this context, it may be relevant that the system has a scaling:
for every solution $\bar{u},\io$, and every $\lambda\in\mathbb{R}$,
\[
\bar{u}_{\lambda}(s,x)=\lambda\bar{u}(\lambda^{2}s,\lambda x),\:(\io)_{\lambda}(s)=\lambda^{2}\io(\lambda^{2}s)
\]
also gives a solution. 
\end{rem}
\strut
\begin{rem}
The extension of Theorem \ref{thm:MR} to the case of external forcing
$F\neq0$, which could for example be given by a gravitational field,
is open. We expect the result to still hold if the forcing vanishes
suitably as $t\to\infty$, e.g. $F\in\mathcal{V}^{\infty}\cap L^{1}(0,\infty;L^{2}(\mathcal{F})\times\mathbb{R}^{6})$.
\end{rem}
\strut

\begin{acknowledgement*}
The author would like to thank Professor Paolo Galdi for introducing
her to this problem and for many inspiring discussions. The author
would like to thank Professor Alexander Mielke for many helpful conversations
and remarks.
\end{acknowledgement*}

\address{Weierstrass Institute, Mohrenstr. 39, 10117 Berlin}

\email{karoline.disser@wias-berlin.de}

\appendix
\bibliographystyle{plain}
\bibliography{/Users/Karo/Dropbox/Arbeit/MR}

\begin{thebibliography}{10}

\bibitem{Amann95}
H.~Amann.
\newblock {\em Linear and {Q}uasilinear {P}arabolic {P}roblems. Vol.~{I}},
  volume~89 of {\em Monographs in Mathematics}.
\newblock Birkh\"{a}user, Boston, 1995.

\bibitem{CazenaveHaraux1998}
Thierry Cazenave and Alain Haraux.
\newblock {\em An introduction to semilinear evolution equations}, volume~13 of
  {\em Oxford Lecture Series in Mathematics and its Applications}.
\newblock The Clarendon Press, Oxford University Press, New York, 1998.
\newblock Translated from the 1990 French original by Yvan Martel and revised
  by the authors.

\bibitem{CumsilleTakahashi08}
P.~Cumsille and T.~Takahashi.
\newblock Wellposedness for the system modelling the motion of a rigid body of
  arbitrary form in an incompressible viscous fluid.
\newblock {\em Czechoslovak Math. J.}, 58 (133):961--992, 2008.

\bibitem{Feireisl03}
E.~Feireisl.
\newblock On the motion of rigid bodies in a viscous compressible fluid.
\newblock {\em Arch. Ration. Mech. Anal.}, 167:281--308, 2003.

\bibitem{Galdi02}
G.~P. Galdi.
\newblock On the motion of a rigid body in a viscous liquid: a mathematical
  analysis with applications.
\newblock In S.~J. Friedlander and D.~Serre, editors, {\em Handbook of
  mathematical fluid dynamics, {V}ol. {I}}, pages 653--791. North-Holland,
  Amsterdam, 2002.

\bibitem{GaldiMazzoneZunino2013}
G.~P. Galdi, G.~Mazzone, and P.~Zunino.
\newblock Inertial motions of a rigid body with a cavity filled with a viscous
  liquid.
\newblock {\em C. R. Mecanique}, 341:760Ð--765, 2013.

\bibitem{GaldiSilvestre02}
G.~P. Galdi and A.~L. Silvestre.
\newblock Strong solutions to the problem of motion of a rigid body in a
  {N}avier-{S}tokes liquid under the action of prescribed forces and torques.
\newblock In {\em Nonlinear Problems in Mathematical Physics and Related Topics
  I}, pages 121--144. Kluwer Academic/Plenum Publishers, New York, 2002.

\bibitem{GaldiLecture2000}
Giovanni~P. Galdi.
\newblock An introduction to the {N}avier-{S}tokes initial-boundary value
  problem.
\newblock In {\em Fundamental directions in mathematical fluid mechanics}, Adv.
  Math. Fluid Mech., pages 1--70. Birkh\"auser, Basel, 2000.

\bibitem{GGH10a}
M.~Geissert, K.~G{\"o}tze, and M.~Hieber.
\newblock ${L}\sp p$-theory for strong solutions to fluid rigid-body
  interaction in {N}ewtonian and generalized {N}ewtonian fluids.
\newblock {\em Trans. Amer. Math. Soc.}, 365(3):1393--1439, 2013.

\bibitem{GreenspanHoward1963}
H.~P. Greenspan and L.~N. Howard.
\newblock On a time-dependent motion of a rotating fluid.
\newblock {\em J. Fluid Mech.}, 17:385--404, 1963.

\bibitem{KopachevskyKrein2003}
Nikolay~D. Kopachevsky and Selim~G. Krein.
\newblock {\em Operator approach to linear problems of hydrodynamics. {V}ol.
  2}, volume 146 of {\em Operator Theory: Advances and Applications}.
\newblock Birkh\"auser Verlag, Basel, 2003.
\newblock Nonself-adjoint problems for viscous fluids.

\bibitem{Leray1934b}
Jean Leray.
\newblock Sur le mouvement d'un liquide visqueux emplissant l'espace.
\newblock {\em Acta Math.}, 63(1):193--248, 1934.

\bibitem{Lyashenko1998}
A.~A. Lyashenko.
\newblock Instability of a precessing body with a fluid-filled cavity.
\newblock In {\em Nonlinear instability, chaos and turbulence, {V}ol.\ {I}},
  volume~20 of {\em Adv. Fluid Mech.}, pages 197--224. WIT Press/Comput. Mech.
  Publ., Boston, MA, 1998.

\bibitem{Lyashenko1998B}
A.~A. Lyashenko.
\newblock Instability of a precessing body with a fluid-filled cavity.
\newblock In {\em Nonlinear instability, chaos and turbulence, {V}ol.\ {I}},
  volume~20 of {\em Adv. Fluid Mech.}, pages 197--224. WIT Press/Comput. Mech.
  Publ., Boston, MA, 1998.

\bibitem{LyashenkoFriedlander1998}
Andrei~A. Lyashenko and Susan~J. Friedlander.
\newblock Nonlinear instability of a precessing body with a cavity filled by an
  ideal fluid.
\newblock {\em SIAM J. Math. Anal.}, 29(3):600--618 (electronic), 1998.

\bibitem{LyashenkoFriedlander1998B}
Andrei~A. Lyashenko and Susan~J. Friedlander.
\newblock A sufficient condition for instability in the limit of vanishing
  dissipation.
\newblock {\em J. Math. Anal. Appl.}, 221(2):544--558, 1998.

\bibitem{MarsdenRatiu1999Book}
Jerrold~E. Marsden and Tudor~S. Ratiu.
\newblock {\em Introduction to mechanics and symmetry}, volume~17 of {\em Texts
  in Applied Mathematics}.
\newblock Springer-Verlag, New York, second edition, 1999.
\newblock A basic exposition of classical mechanical systems.

\bibitem{MoiseyevRumyantsev1968}
N.~N. Moiseyev and V.~V. Rumyantsev.
\newblock {\em Dynamic Stability of Bodies Containing fluid}, volume~6 of {\em
  Applied Physics and Engineering}.
\newblock Springer, New York, 1968.

\bibitem{Poincare1910}
M.~H. Poincare.
\newblock Sur la precession des corps deformables.
\newblock {\em Bull. Astronomique}, 27:321Ð--356, 1910.

\bibitem{SilvestreTakahashi2012}
Ana~L. Silvestre and Tak{\'e}o Takahashi.
\newblock On the motion of a rigid body with a cavity filled with a viscous
  liquid.
\newblock {\em Proc. Roy. Soc. Edinburgh Sect. A}, 142(2):391--423, 2012.

\bibitem{Solonnikov77}
V.~A. Solonnikov.
\newblock Estimates for solutions of nonstationary {N}avier-{S}tokes equations.
\newblock {\em J. Sov. Math.}, 8:467--529, 1977.

\bibitem{StewartsonRoberts1963}
K.~Stewartson and P.~H. Roberts.
\newblock On the motion of a liquid in a spheroidal cavity of a precessing
  rigid body.
\newblock {\em J. Fluid Mech.}, 17:1--20, 1963.

\bibitem{Takahashi03}
T.~Takahashi.
\newblock Analysis of strong solutions for the equations modeling the motion of
  a rigid-fluid system in a bounded domain.
\newblock {\em Adv. Differential Equations}, 8(12):1499--1532, 2003.

\bibitem{Zhukovskiy1885}
N.~Y. Zhukovskiy.
\newblock On the motion of a rigid body with cavities filled with a homogeneous
  liquid drop.
\newblock {\em Zh. Fiz.-Khim. Obs. Physics}, 17:81--113, 1885.

\end{thebibliography}

\end{document}